%% file: arxiv.tex
\documentclass[11pt,a4paper]{article}

\usepackage[T1]{fontenc}
\usepackage[utf8]{inputenc}
\usepackage[a4paper,margin=1in]{geometry}

\usepackage{graphicx}
\graphicspath{{figures/}}
\usepackage{subcaption}

\usepackage{tabularx}
\usepackage{multirow}
\usepackage{makecell}
\usepackage{diagbox}
\usepackage{booktabs}

\usepackage{amsmath,amssymb,amsfonts}
\usepackage{amsthm}
\usepackage{mathrsfs}
\usepackage{upgreek}
\usepackage{bm}
\usepackage{dsfont}
\usepackage{bbold}
\usepackage{accents}

\usepackage[numbers]{natbib}
\usepackage[ruled]{algorithm}
\usepackage{algpseudocode,algorithmicx}

\usepackage[title]{appendix}
\usepackage[table,xcdraw]{xcolor}
\usepackage{textcomp}
\usepackage{manyfoot}
\usepackage{listings}
\usepackage{comment}
\usepackage{enumitem}
\setlist[enumerate]{itemsep=0mm}
\usepackage{lscape}
\usepackage{pdflscape}

\newtheorem{theorem}{Theorem}
\newtheorem{proposition}{Proposition}
\newtheorem{lemma}{Lemma}

\theoremstyle{definition}
\newtheorem{definition}{Definition}
\newtheorem{assume}{Assumption}
\newtheorem{remark}{Remark}
\newtheorem{example}{Example}
\usepackage{tikz,pgfplots,pgfplotstable}

\pgfplotsset{compat=1.18}

\usepgfplotslibrary{groupplots,fillbetween}
\usetikzlibrary{
    intersections,
    calc,
    arrows.meta,
    matrix,
    spy,
    positioning
}

\definecolor{darkred}{rgb}{0.6,0,0}
\definecolor{darkgreen}{rgb}{0,0.5,0}
\definecolor{darkblue}{rgb}{0,0,0.5}
\definecolor{SkyBlue}{rgb}{0.53,0.81,0.92}

\usepackage[
    colorlinks=true,
    linkcolor=blue,
    citecolor=blue,
    urlcolor=blue
]{hyperref}

\input{newcommand}

\title{Relax and Follow: $\ell_0$-Path Computation with $\ell_0$-Bregman Relaxations}

\author{
Mhamed Essafri\\
\small Univ. Toulouse, Toulouse INP, IRIT, France\\
\small \texttt{mhamed.essafri@irit.fr}
\and
Luca Calatroni\\
\small MaLGa, DIBRIS, Università di Genova\\
\small Istituto Italiano di Tecnologia, Genoa, Italy\\
\small \texttt{luca.calatroni@unige.it}
\and
Emmanuel Soubies\\
\small Univ. Toulouse, Toulouse INP, CNRS, IRIT, France\\
\small \texttt{emmanuel.soubies@cnrs.fr}
}

\date{}

\begin{document}

\maketitle

\begin{abstract}
This work introduces \Algo{}, a novel method for estimating the solution path of $\ell_0$-regularized problems through the use of $\ell_0$ Bregman relaxations (\BR{}). Recently introduced and analyzed in the literature, these relaxations provide continuous reformulations of the original objective, are applicable to possibly non-quadratic data fidelity terms, and depend on a family of functions designed to preserve the global minimizers while eliminating part of the undesirable local minima. Given any numerical solver for the relaxation, the proposed approach dynamically constructs a collection of local minimizers that are candidates for the $\ell_0$-solution path. It exploits warm-start strategies and identifies ranges of the regularization parameter for which each minimizer remains valid under the corresponding relaxation.
Experiments on sparse least-squares and logistic regression problems demonstrate that \Algo{} systematically outperforms state-of-the-art baselines across both synthetic and real-world datasets in terms of various evaluation metrics; additionally, the study investigates how the choice of the \BR{} affects the quality of the estimated path in the sparse Poisson regression setting.
\end{abstract}

\vspace{0.5em}

\section{Introduction}

Sparse optimization plays a central role in many fields, including machine learning, statistics, and signal processing, see, e.g., ~\citep{tillmann2024cardinality} for a review.
In this work, we focus on underdetermined regimes in which the number of unknown parameters, \( N \in \N \), exceeds the sample size, \( M\in\N \) (\( N \gg M \)). In such settings, sparsity becomes particularly appealing, as it enables the construction of low-complexity and interpretable models. Specifically, we consider optimization problems of the form:
\begin{equation}\label{eq:problem_setting}
\hat{\x} \in \argmin_{\x \in \Cc^N} J_0(\x), \quad \text{with} \quad J_0(\x) := F_\y(\A \x) + \lambda_0 \|\x\|_0 + \frac{\lambda_2}{2} \|\x\|^2_2,
\end{equation}
where  \( \A \in \R^{M \times N} \) denotes the measurement (or forward/feature) matrix, and \( \y \in \mathcal{Y}^M \) is the vector of observations, with  $\mathcal{Y} \subseteq \R$ depending on the considered problem. The first term, \( F_\y :
\R^M  \mapsto \R_{\geq 0}\), corresponds to the data fidelity, which measures the discrepancy between the model $\A\x$ and the observed data $\y$. The second term, \( \|\cdot\|_0 \), is the \( \ell_0 \) pseudo-norm that counts the number of nonzero components in its argument, thereby promoting sparsity. The third term, \( \|\cdot\|^2_2 \), is a ridge regularizer, which can penalize large coefficients, reduce overfitting in noisy settings, and enforce the the well-posedness of the optimization problem (i.e., guarantee the existence of a solution) for certain data fidelity terms. We consider the case where the unknown $\x$ is constrained to lie in the set $\Cc^N$ for which, in this paper, we assume either $\Cc = \R$ (no constraint) or $\Cc = \R_{\geq 0}$, the nonnegative orthant. The hyperparameters \( \lambda_0 > 0 \) and \( \lambda_2 \geq 0 \) balance the trade-off between data fidelity, sparsity, and ridge penalization (when present). We assume that $F_\y$ satisfies the following assumption.

\begin{assume}\label{ass:Fy}
 The data fidelity $F_\y$ is coordinate-wise separable, i.e., there exists $f : \R \times \mathcal{Y} \to \R_{\geq 0}$ such that,  for all $\z \in \R^M$, $$F_\y(\z) = \sum_{m=1}^M f(z_m;y_m).$$ Moreover, for each $y \in \mathcal{Y}$, we assume that $f(\cdot;y)$ is convex, proper, twice differentiable on its domain and bounded from below.
\end{assume}

Assumption~\ref{ass:Fy} is satisfied by a broad class of commonly used data-fidelity terms, including the least-squares criterion, the generalized Kullback-Leibler divergence for count data, and the logistic loss for binary classification.

\subsection{Related Works} 

Although the $\ell_0$ pseudo-norm is a natural choice for promoting sparsity, its combinatorial nature leads to  NP-hard optimization problems~\citep{ Natarajan1995-np-hard,nguyen-np-hard} which limits its use in practice. This difficulty has motivated a substantial body of work~\citep{tillmann2024cardinality} on tractable approximations of Problem \eqref{eq:problem_setting} as well as the development of algorithmic strategies, which we briefly review in the first paragraph below. Moreover, selecting an appropriate value for the sparsity-inducing parameter $\lambda_0$ is itself a nontrivial task. In practice, $\lambda_0$ is difficult to tune and highly problem-dependent, which has motivated the development of methods that explore the entire solution path over a range of regularization parameters. We review these approaches in the last two paragraphs of this section, distinguishing between methods that rely on a prescribed grid of $\lambda_0$ values and those that automatically identify the critical values of $\lambda_0$ at which changes occur.

\paragraph{Existing algorithms and continuous relaxations.}
{First, by reformulating Problem~\eqref{eq:problem_setting} as a mixed integer program (MIP), global solutions can be obtained using branch-and-bound (BnB) techniques for problems of moderate size~\citep{bourguignon2015exact,bertsimas2016best,guyard:24,delle2023novel}. The MIP formulation also enables the use of perspective relaxation techniques to convexify products between continuous variables and binary  variables~\citep{frangioni2006perspective,gunluk2010perspective,wei2022ideal}.} {Another class of approaches is the one of} greedy algorithms such as matching pursuit (MP)~\citep{Mallat}, orthogonal matching pursuit (OMP)~\citep{Pati1993Orthogonal}, and single best replacement (SBR)~\citep{Soussen2011FromBD}. They  iteratively modify the support of the solution (e.g., adding, removing or swapping components) according to a given rule. Alternatively, a large body of works replace the discontinuous $\ell_0$ pseudo-norm with continuous convex or nonconvex surrogates. Among convex choices, the most popular one is the $\ell_1$-norm~\citep{tibshirani1996regression,Candes} whose use has been popularized thanks to the development of compressed sensing methods. In the non-convex setting, a broad range of penalties has been proposed, including the capped-$\ell_1$~\citep{zhang2008multi}, $\ell_p$ quasi-norms $(0<p<1)$~\citep{Foucart2009SparsestSO}, the log-sum penalty~\citep{Cands2007EnhancingSB}, the smoothed $\ell_0$ penalty~\citep{Mohimani2008AFA}, the smoothly clipped absolute deviation (SCAD)\citep{Fan2001VariableSV}, the minimax concave penalty (MCP)~\citep{Cun-Hui}, exponential approximations~\citep{Mangasarian1996}, norms ratio $\ell_p/\ell_q$~\citep{Repetti2015,cherni2020spoq}, as well as the reverse Huber penalty~\citep{Pilanci}. These formulations aim to better approximate the $\ell_0$ pseudo-norm while alleviating the bias introduced by convex surrogates. 
Relaxations preserving the set of global minimizers of the original $\ell_0$ problem without introducing spurious local minima (that is, \emph{exact} continous relaxations) have also been introduced. Examples include the continuous exact $\ell_0$ penalty (CEL$0$)~\citep{Soubies2015,soubies2017unified}, which is a special case of the quadratic envelope~\citep{Carlsson2019}. These ideas have been generalized to non-quadratic data terms through the capped-$\ell_1$ penalty~\citep{Weigeneral} and the $\ell_0$ Bregman relaxations (\BR{})~\citep{essafri2024exact} which are the ones considered in this work.

\paragraph{Regularization paths with a predefined $\lambda_0$ grid.} To obtain an estimate of the regularization path, a simple strategy is to apply algorithms designed for a fixed value of $\lambda_0$ along a predefined grid of regularization parameters using warm-start techniques. For instance, \texttt{SparseNet}~\citep{sparseNet:12}, \texttt{ncvreg}~\citep{Breheny2011Ncvreg}, and \texttt{skglm}~\citep{skglm} adopt this strategy to compute the regularization paths of (non-convex) continuous relaxations such as MCP and SCAD, while \texttt{El0ps}~\citep{guyard2025el0ps} tackles the computation of the $\ell_0$ regularization path directly by solving the problem via BnB for each value of $\lambda_0$ on the grid.


\paragraph{Regularization paths with automatic computation of critical $\lambda_0$ values.} Other approaches construct the regularization path by automatically identifying the values of $\lambda_0$ at which the solution changes. {In the context of convex relaxations (e.g., $\ell_1$, $\ell_2$, and elastic net) these techniques have been extensively studied, leading to the development of efficient algorithms~\citep{efron2004least,rosset2007piecewise,friedman2007pathwise,Friedman2010Glmnet}. Their general idea is to exploit the fact that the coefficient profiles of the solution vector are piecewise linear with respect to the regularization parameter. Surprisingly, less attention has been paid to the study of $\ell_0$ regularization path of Problem~\eqref{eq:problem_setting}.}
Yet,~\cite{SoussenPath:15} exploited the piecewise-linear and concave nature of the so-called $\ell_0$-curve (see Section~\ref{sec:l0-curve-background}) to develop two greedy methods, named \texttt{CSBR} and \texttt{L0PD}, that iteratively refine both critical $\lambda_0$ values and the corresponding estimates of the solution (thus providing an estimate of the $\ell_0$-path). {In a different vein,~\cite{CWLSHazimeh:20} proposed the \Algolearn{} method~\citep{hazimeh2023l0learn} to estimate the $\ell_0$-path using coordinate descent (CD) on $J_0$ combined with a local combinatorial search. To compute the sequence of critical $\lambda_0$ values, they exploit the fact that CD stationary points are coordinate-wise (CW) minimizers. Specifically, they derive}
 lower bounds on  $\lambda_0$ for which a CW minimizer stops being a CW minimizer. 
 Hence, given a CW minimizer and its lower bound on $\lambda_0$, the next $\lambda_0$ is chosen slightly below this bound, with the current CW minimizer used as a warm-start point. Initially developed for least-squares regression, this framework was later extended to classification problems, including logistic regression, by~\cite{HazimehClass}.

\subsection{Contributions and Outline}

In this work, we leverage the good properties of exact $\ell_0$ Bregman relaxations of the functional $J_0$,  denoted $J_\Psi$, to derive \Algo{}, a method estimating the solution path of Problem~\eqref{eq:problem_setting} that can be deployed with any solver capable of minimizing~$J_\Psi$. 

We first provide background material in Section~\ref{sec:prelim}, covering useful properties of $J_0$ (and associated $\ell_0$-path) as well as of $J_\Psi$ (definition, exact relaxation property, characterization of local minimizers).

We then derive the proposed \Algo{} method in Section~\ref{sec:main_alg}. Its effectiveness relies on Theorem~\ref{th:min-loc-removed}, which, for each local minimizer $\hat{\x}$ of the exact relaxation $J_\Psi$, provides an interval of values for $\lambda_0$ over which $\hat{\x}$ remains a local minimizer. The principle of \Algo{} (Algorithm~\ref{algo:proposed}) is then to exploit these $\lambda_0$ intervals to implement warm-start strategies  within a tree-search framework. Specifically, the method dynamically builds a set of local minimizers that are candidate to be in the $\ell_0$-path (initialized with the zero vector). This is achieved through multiple passes over points already in this set, using them as warm-starting points for the inner solver with $\lambda_0$ slightly larger (resp., smaller) than the upper (resp., lower) bound of the regularization parameter range associated with the current explored point. After a given number of passes or when no new points remain to be explored in the set,  a final estimate of the $\ell_0$-path is extracted from the computed candidate local minimizers.

Finally, in Section~\ref{sec:expe}, we benchmark \Algo{} against \Algolearn{} on sparse least-squares and logistic regression problems, using both synthetic and real-world datasets. Evaluations are performed with both statistical and optimization-based metrics. We also provide complementary comparisons with \texttt{L0PD} for some of these experiments. In all reported results, \Algo{} 
systematically achieves superior metrics compared to the baselines, independently on the inner solver used. To conclude, we also consider sparse Poisson regression experiments (involving Kullback-Leibler data fidelity) to compare different $J_\Psi$, which are known to eliminate different local minimizers of $J_0$.

We have developed \Algo{} as a modular Python package that supports various combinations of data-fidelity terms, exact relaxations, and inner optimization algorithms. The package will be made publicly available upon acceptance of this work.

\subsection{Notations}

We use the following notation throughout the paper. The set of nonnegative real numbers is denoted by $\R_{\geq 0} = \{x \in \R : x \geq 0\}$. The identity matrix of size $N$ is denoted by $\mathbf{I} \in \R^{N \times N}$. We denote by $[N]$ the set $\{1, 2, \ldots, N\}$ of the first $N$ natural numbers. For a vector $\x \in \R^N$, we define $\x^{(n)} = (x_1, \dots, x_{n-1}, 0, x_{n+1}, \dots, x_N)$ as the vector obtained by zeroing the $n$th coordinate. The support of $\x$, denoted by $\sigma(\x) \subseteq [N]$, is the set of indices $n$ such that $x_n \neq 0$. For a set $\omega \subseteq [N]$, we write $\sharp \omega$ for its cardinality. The submatrix of $\A \in \R^{M \times N}$ formed by selecting only the columns indexed by $\omega$ is denoted by $\A_\omega = (\a_{\omega[1]}, \dots, \a_{\omega[\sharp \omega]}) \in \R^{M \times \sharp \omega}$, and the corresponding restriction of a vector $\x \in \R^N$ is written as $\x_\omega = (x_{\omega[1]}, \dots, x_{\omega[\sharp \omega]}) \in \R^{\sharp \omega}$. The $m$th row and $n$th column of $\A \in \R^{M \times N}$ are denoted respectively $\underline{\a}_m \in \R^N$ and $\a_n \in \R^M$. The unit vector corresponding to the $n$th canonical basis vector of $\R^N$ is denoted by $\e_n$. Unless otherwise stated, $\|\cdot\|$ refers to the standard Euclidean norm $\|\cdot\|_2$.

\section{Preliminaries}\label{sec:prelim}

This section provides the necessary background on the $\ell_0$-path and the \BR{} penalty, which yields an exact continuous relaxation of Problem~\eqref{eq:problem_setting}.

\subsection{Background on $\ell_0$-Path}\label{sec:l0-curve-background}


We recall the definition and the main properties of the \(\ell_0\)-path, which refers to the set of global solutions of Problem~\eqref{eq:problem_setting} as a function of \(\lambda_0\).

To begin, let us recall the characterization of the local minimizers of $J_0$~\citep{Thi2015,essafri2024exact}.
\begin{proposition} A point $\hat \x \in \Cc^N$ is a local minimizer of $J_0$ in \eqref{eq:problem_setting} if and only if it solves
\begin{equation}\label{eq:locMinJ0}
    \hat{\x}_{\hat \sigma} \in \argmin_{\z \in \Cc^{\sharp \hat{\sigma}}} F_\y(\A_{\hat \sigma}\z) + \frac{\lambda_2}{2} \|\z\|^2
\end{equation}
where $\hat{\sigma} = \sigma(\hat \x)$ stands for the support of $\hat \x$. 
Moreover, $\hat \x$ is strict (i.e., there exists a neighbourhood $\mathcal{V}$ of $\hat{\x}$ such that $\forall \x \in \mathcal{V} \backslash \{\hat \x\}$, $J_0(\hat \x) < J_0(\x)$) if and only if $\lambda_2 >0$ or $\A_{\hat{\sigma}}$ is full rank. Finally, global minimizers are strict.
\end{proposition}

A key observation is  that~\eqref{eq:locMinJ0} is independent of \(\lambda_0\). Consequently, the set of local minimizers of \(J_0\) remains unchanged with respect to \(\lambda_0\). Then, given a local minimizer \(\hat{\x}\) of \(J_0\), the objective value \(J_0(\hat{\x})\) is an affine function of \(\lambda_0\) on \(\mathbb{R}_{\geq 0}\), with slope  \(\|\hat{\x}\|_0\) and  intercept  \((F_\y(\A\hat{\x}) + \frac{\lambda_2}{2}\|\hat{\x}\|^2)\),
\begin{equation}\label{eq:affine-lam0}
    \lambda_0 \mapsto  (F_\y(\A\hat{\x}) + \frac{\lambda_2}{2}\|\hat{\x}\|^2) + \lambda_0 \, \|\hat{\x}\|_0.
\end{equation}
An illustration is provided in Figure~\ref{fig:l0curve-original}, where each affine function corresponds to a strict local minimizer of $J_0$ for a problem of size \((3,4)\). The lower concave envelope of these affine functions is known as the \textit{\(\ell_0\)-curve}~\citep{SoussenPath:15} and is associated to points of the \textit{$\ell_0$-path}.  The breakpoints (denoted \(\hat{\lambda}_0^i, \; i \in \{1,2,3\}\), in Figure~\ref{fig:l0curve-original}) correspond to the  \(\lambda_0\) values at which the support of the global minimizer changes. While the $\ell_0$-curve is piecewise affine, the $\ell_0$-path is piecewise constant with respect to \(\lambda_0\)~\citep{SoussenPath:15}. One can observe that all local minimizers \(\hat{\x}\) of \(J_0\) satisfying \(\|\hat{\x}\|_0 = 3\) minimize the data-fidelity term alone, as the corresponding sub-matrices $3\times 3$ are full-rank. Consequently, the corresponding affine functions in the \((\lambda_0, J_0)\)-plane are identical. This behaviour was analysed in detail by~\citet[Section 3.3]{nikolova2013description} for the least-squares case.

The affine and constant piecewise structure of the $\ell_0$-curve and $\ell_0$-path, as established in~\citep{SoussenPath:15} for the least-squares setting, extends naturally to our more general framework with non-quadratic fidelities.

\begin{figure}[t]
    \centering
    \includegraphics[width=0.9\textwidth]{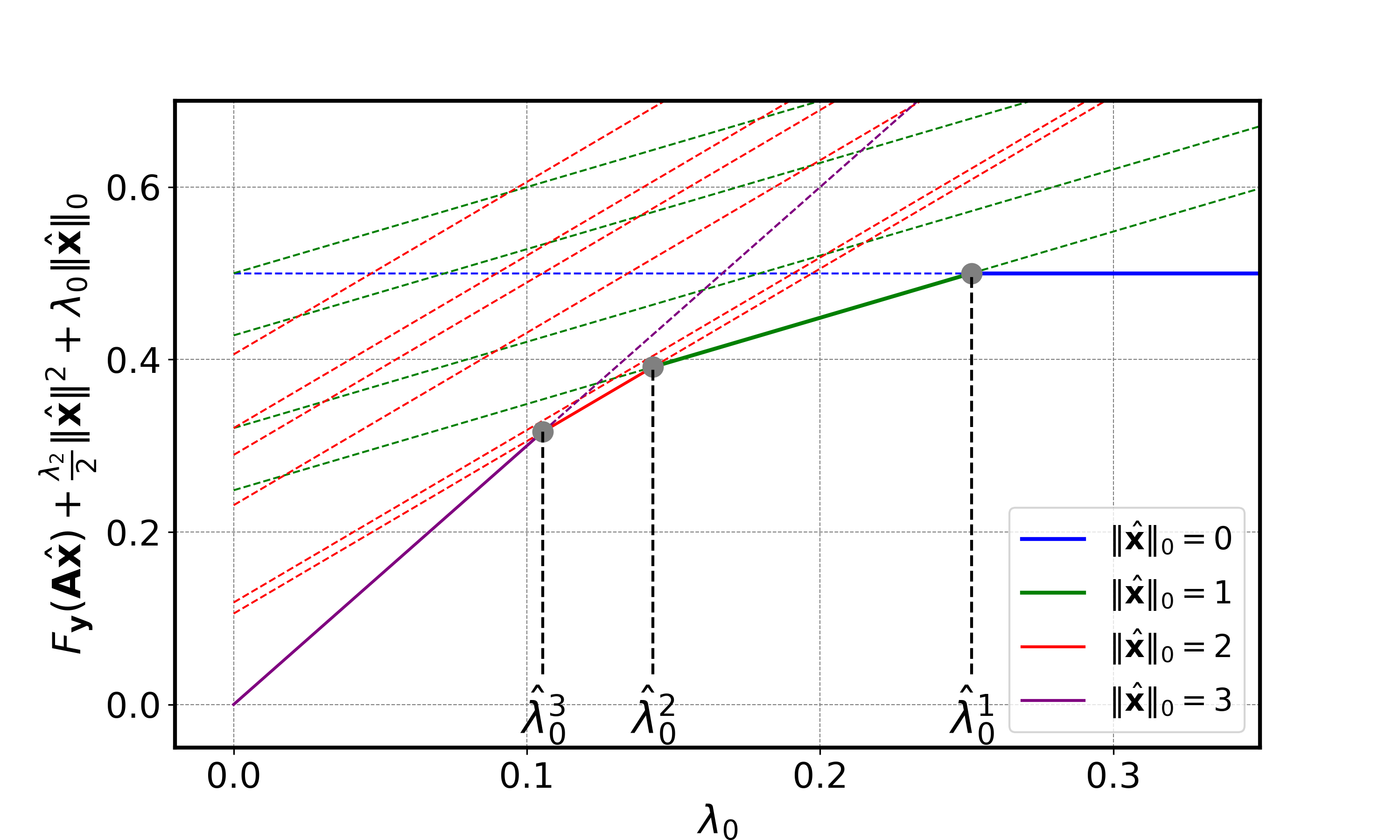}
\caption{The $\ell_0$-curve. Plot of the affine functions
$\lambda_0 \mapsto (F_\y(\A\hat{\x}) + \frac{\lambda_2}{2}\|\hat{\x}\|^2) + \lambda_0 \|\hat{\x}\|_0$
for all fifteen (strict local) minimizers $\hat{\x}$ of a $\ell_0$-regularized least-squares problem ($F_\y(\A \cdot) = \tfrac{1}{2} \|\A\cdot - \y\|^2$, $\lambda_2=0$) of size $(M, N) = (3,4)$. The $\ell_0$-curve corresponds to the lower concave envelope of these affine functions, and is represented by the solid line. It is piecewise linear, with breakpoints $(\hat{\lambda}_0^1, \hat{\lambda}_0^2, \hat{\lambda}_0^3)$ that delineate the intervals over which the global minimizer remains unchanged ($\ell_0$-path).
}\label{fig:l0curve-original}
\end{figure}

\subsection{Background on $\ell_0$ Bregman Relaxations}\label{sec:Brex}

We now review the \(\ell_0\) Bregman relaxations (\BR{}) introduced in~\citep{essafri2024exact}. This relaxation replaces the discontinuous \(\ell_0\) term in Problem~\eqref{eq:problem_setting} with a nonconvex continuous penalty,  preserving all global minimizers of \(J_0\), while removing some local ones. 

\begin{definition}[$\ell_0$ Bregman relaxation]\label{def:brex}
Consider a family of scalar functions  $\Psi = \left\lbrace \psi_n : \Cc \mapsto \R \right\rbrace_{n \in [N]}$ such that 
\begin{enumerate}[label=\roman*.]
    \item $\psi_n$ is strictly convex, proper, and twice-differentiable over $\operatorname{int}(\Cc)$,
    \item the map $x \mapsto \psi_n'(x)x - \psi_n(x)$ is coercive,
     \item when $\Cc = \R$, $\psi_n'$ is odd.
\end{enumerate}
 Then, the \emph{\BR{} penalty} is defined as
\begin{equation}\label{eq:brex-original-formula}
 B_\Psi(\x): = \sup_{\alpha\in \mathbb{R}} \; \sup_{\z \in \operatorname{int}(\Cc^N)}~ \left\lbrace \alpha- D_{\Psi}(\x,\z): \alpha -  D_{\Psi}(\cdot,\z) \leq \lambda_0 \|\cdot\|_0  \right\rbrace
 ,
\end{equation}
where $D_{\Psi} : \Cc^N \times \mathrm{int}(\Cc^N) \to \R_{\geq 0}$ is the separable Bregman divergence~\citep{Bregman1967} associated with $\Psi$,
\begin{equation}\label{eq:bregman-distance} D_\Psi(\x,\z) = \sum_{n=1}^N d_{\psi_n}(x_n,z_n) \quad \text{ with } \quad d_{\psi_n}(x,z) = \psi_n(x) - \psi_n(z) - \psi_n'(z) (x-z)
.
\end{equation}
\end{definition}


As shown in \citep[Proposition 5]{essafri2024exact},  the \BR{} penalty is separable, $B_\Psi(\x) = \sum_{n=1}^N \beta_{\psi_n}(x_n)$, and we can derive a closed-form expression of the scalar functions $\beta_{\psi_n}$. For all $x \in \Cc$
\begin{equation}\label{eq:brex-generic-formula} 
 \beta_{\psi_n}(x) = \left\lbrace 
    \begin{array}{ll}
        \psi_n(0) - \psi_n(x) +  \psi_n'(\alpha_n)|x| \quad &\quad  \text{ if } |x| \leq \alpha_n, \\
      \lambda_0  \quad  & \quad  \text{ otherwise.}
    \end{array}
    \right.
\end{equation}
where $\alpha_n > 0$ is the unique (see Lemma~\ref{lem:alpha-increasing} in the Appendix) positive point satisfying $$\lambda_ 0 = \psi_n(0) - \psi_n(\alpha_n) + \psi'_n(\alpha_n)\alpha_n \;\, \left( = d_{\psi_n}(0, \alpha_n)\right).$$
This parameter can be computed explicitly for several common generating functions, as detailed in~\citep[Table~3]{essafri2024exact}.

\begin{remark}
    Condition iii. in Definition~\ref{def:brex} can be relaxed. Without this condition, \BR{} loses its symmetry, requiring the introduction of two distinct parameters \(\alpha_n^- \leq 0\) and \(\alpha_n^+ \geq 0\), as done in~\citep{essafri2024exact}. However, since symmetric penalties are typically considered in practice, we include Condition iii. in this work to simplify the presentation.
\end{remark}

Replacing the $\ell_0$ pseudo-norm in \eqref{eq:problem_setting} with the penalty $B_\Psi$ yields the following continuous relaxation of $J_0$:
\begin{equation}\label{eq:relaxation}
    J_\Psi(\x) = F_\y(\A\x) + B_\Psi(\x)  + \frac{\lambda_2}{2}\|\x\|^2.
\end{equation}
For suitable choices of $\Psi$, this relaxation is exact in the sense given by the following theorem.

\begin{theorem}[Theorem 9 in \citet{essafri2024exact}]\label{th:exact}
    Let $J_\Psi$ be defined as in~\eqref{eq:relaxation}. If for all $n \in[N]$ and $\x \in \Cc^N$ the following condition holds : $\forall t \in (-\alpha_n, 0) \cup (0, \alpha_n)$,
\begin{equation}\label{eq:cc-condition}
\frac{\partial^2}{\partial t^2 } J_\Psi(\x^{(n)}+t\e_n)= 
\sum_{m \in [M]} a_{mn}^2 f''([\A \x^{(n)}]_m + ta_{mn};y_m) 
 -\psi_n''(t) + \lambda_2 < 0 \;  \; 
 \tag{CC}
\end{equation}
then~$J_\Psi$ is an \emph{exact relaxation} of $J_0$, i.e., over $\Cc^N$, the set of global minimizers of $J_\Psi$ and $J_0$ coincide, and local minimizers of $J_\Psi$ are local minimizers of $J_0$.
\end{theorem}

From Theorem~\ref{th:exact}, we deduce that \(J_\Psi\) may eliminate local (not  global) minimizers of \(J_0\), while preserving all global ones.

To conclude this section, we recall in Propositions~\ref{prop:crit-pts-jpsi} and~\ref{prop:loc-min-jpsi} the characterizations of both critical points and local minimizers of $J_\Psi$ given in~\citep[Proposition 7 and Corollary 3]{essafri2024exact}.

\begin{proposition}[Critical points of $J_\Psi$]\label{prop:crit-pts-jpsi} 
   A  point $\hat{\x} \in \Cc^N$ is a critical point of $J_\Psi$ \emph{if and only if}, for all $n \in [N]$,  
   \begin{equation}
     \begin{cases}
     \varphi_n\!\left(\psi_n'(0) - \left\langle \a_n , \nabla F_\y(\A \hat{\x}) \right\rangle \right) \leq \psi_n'(\alpha_n)  & \text{if } \hat{x}_n = 0,
 \\
 \left\langle \a_n , \nabla F_\y(\A \hat{\x}) \right\rangle + \lambda_2 \hat{x}_n  - \psi_n'(\hat{x}_n) + \mathrm{sign}(\hat{x}_n) \psi_n'(\alpha_n)= 0 & \text{if }  0 < |\hat{x}_n| \leq  \alpha_n, \\
 \left\langle \a_n , \nabla F_\y(\A \hat{\x}) \right\rangle + \lambda_2 \hat{x}_n = 0 & \text{if } |\hat{x}_n| > \alpha_n,
   \end{cases}
   \end{equation}
   where $\varphi_n = |\cdot|$ when $\Cc = \R$ and $\varphi_n = \max(\cdot,\psi'_n(0))$ when $\Cc = \R_{\geq 0}.$
\end{proposition}

\begin{proposition}[Characterization of local minimizers of $J_\Psi$]\label{prop:loc-min-jpsi}  A point $\hat{\x} \in \Cc^N$ is a local minimizer of $J_\Psi$ \emph{if and only if} it is a critical point of $J_\Psi$ and, for all $n\in \sigma(\hat{\x})$, $|\hat{x}_n| > \alpha_n$. This is equivalent to
\begin{align}
& \forall n \in \sigma^c(\hat{\x}), \quad 
\varphi_n\!\left(\psi_n'(0) - \left\langle \a_n , \nabla F_\y(\A \hat{\x}) \right\rangle \right) 
\leq \psi_n'(\alpha_n), 
\label{eq:condition1} \\
& \forall n \in \sigma(\hat{\x}), \quad 
\begin{cases}
\left\langle \a_n , \nabla F_\y(\A \hat{\x}) \right\rangle + \lambda_2 \hat{x}_n = 0, \\
|\hat{x}_n| > \alpha_n.
\end{cases}
\label{eq:condition2}
\end{align}

\end{proposition}

Beyond these characterizations of critical points and local minimizers, the optimization landscape of $J_\Psi$ has been analysed by~\citet{carlsson2020unbiased} (least-squares case) and \citet{chirinos2025optimization} (general case).

\section{The \Algo{} Algorithm for $\ell_0$-Path Computation}\label{sec:main_alg}

In this section, we present \Algo, an algorithm designed to estimate the $\ell_0$-path by leveraging the favourable properties of \BR{}. To that end, we make the following non-restrictive assumption on the family of functions $\Psi$. Note that this assumption is introduced solely for streamlining the presentation (see Remark~\ref{rem:comment_ass_unif_param_psi}) and it is naturally satisfied by standard choices, as discussed in the examples below.

\begin{assume}\label{ass:unif_param_psi}
There exists a parameter space $\mathcal P$ and a family of relaxations ${\Psi_{\mathbf p}}, {\mathbf p \in \mathcal P}$ such that one can choose $\mathbf p \in \mathcal P$ independently of $\lambda_0$ while ensuring that Condition~\eqref{eq:cc-condition} holds uniformly for all $\lambda_0 \geq 0$.
\end{assume}

\begin{example}\label{ex:PsiL2}
Let $\psi_n(x) = \frac{\gamma_n}{2}x^2$ and {the $f(\cdot;y_m)$ have a Lipschitz continuous derivative with constant $L_m$}. Then  $\mathbf{p} = (\gamma_n)_{n\in [N]} \in \R_{>0}^N$ and a  sufficient condition for~\eqref{eq:cc-condition} is then  given~by $$\gamma_n > \lambda_2 + \sum_{m \in [M]} a_{mn}^2 L_m,$$ which is independent of~$\lambda_0$.
\end{example}

{
\begin{example}\label{ex:PsiKL}
    Let $d_\mathrm{KL}(y;z) = z + y \log(y/z) - y$, $f(z_m;y_m) = d_\mathrm{KL}(y_m; z_m  + b_m) / M$ with $b_m >0$  and  $\psi_n(x) = \gamma_n d_{\mathrm{KL}}(\xi ; c_n x + \xi )$. This choice corresponds to a Kullback-Leibler data fidelity with tailored $\psi_n$ (see Section~\ref{sec:expeKL}). Then $\mathbf{p} = ((\gamma_n)_{n \in [N]}, (c_n)_{n \in [N]}, \xi) \in \R_{>0}^{2N+1}$ and sufficient conditions for~\eqref{eq:cc-condition} are then  given by~\citep{essafri2024l0,chirinos2025optimization}
    $$
        c_n = \min_{m \in \sigma(\a_n)} a_{mn}, \quad \gamma_n > \sum_{m \in [M]} \frac{a_{mn}^2 y_m}{ M c_n^2 \xi}, \quad \xi \leq \min_{m \in [M]} b_m,
    $$
    which are independent of $\lambda_0$.
\end{example}}

\subsection{Revisiting the $\ell_0$-Path: Insights from \BR{}}

As shown in Section~\ref{sec:l0-curve-background}, the set of local minimizers of \(J_0\) is invariant with respect to \(\lambda_0\). Consequently, each local minimizer defines a line in the \((\lambda_0, J_0)\)-plane, as depicted in Figure~\ref{fig:l0curve-original}.
On the other hand, Theorem~\ref{th:exact} and Proposition~\ref{prop:loc-min-jpsi} reveal that the exact relaxation \(J_\Psi\) eliminates some local (not global) minimizers of $J_0$. Thus, the set of local minimizers of \(J_\Psi\)  is a subset of the set of local minimizers of \(J_0\), which can be expected to vary with \(\lambda_0\). This property is formalized in Theorem~\ref{th:min-loc-removed}, which establishes that any local minimizer of \(J_\Psi\) maintains such property for a range of \(\lambda_0\) values.  The proof can be found in Appendix~\ref{app:th-min-loc}.

\begin{theorem}\label{th:min-loc-removed}
Let \( \hat{\x} \in \Cc^N \) be a local minimizer of \( J_\Psi \) and set $g_n =  \psi_n'(0)-\left<\a_n ,\nabla F_\y\left(\A\hat{\x} \right) \right>$. Then, there exist two bounds
   \begin{align}
        & \bar{\lambda}_0(\hat \x) = \min_{n \in \sigma(\hat{\x})} \;d_{\psi_n}(0,|\hat{x}_n|) \qquad \text{(or } + \infty \text{ if } \sigma(\hat{\x}) = \emptyset, \text{i.e., } \hat{\x} = \mathbf{0}\text{),} \label{eq:lamBar} \\
        & \underline{\lambda}_0(\hat \x) = \max_{n \in \sigma^c(\hat{\x})} \; d_{\psi_n}\left(0,(\psi_n')^{-1} \left(\varphi_n\left( g_n\right) \right) \right) \qquad \text{(or } 0 \text{ if } \sigma^c(\hat{\x}) = \emptyset\text{),} \label{eq:lamUbar}
    \end{align}
    such that  for all \( \lambda_0 \in [\underline{\lambda}_0(\hat \x), \bar{\lambda}_0(\hat \x)) \), \( \hat{\x} \) is \emph{still} a local minimizer of \( J_\Psi \).  
\end{theorem}

\begin{figure}[!t]
    \centering
    \includegraphics[width=0.9\textwidth]{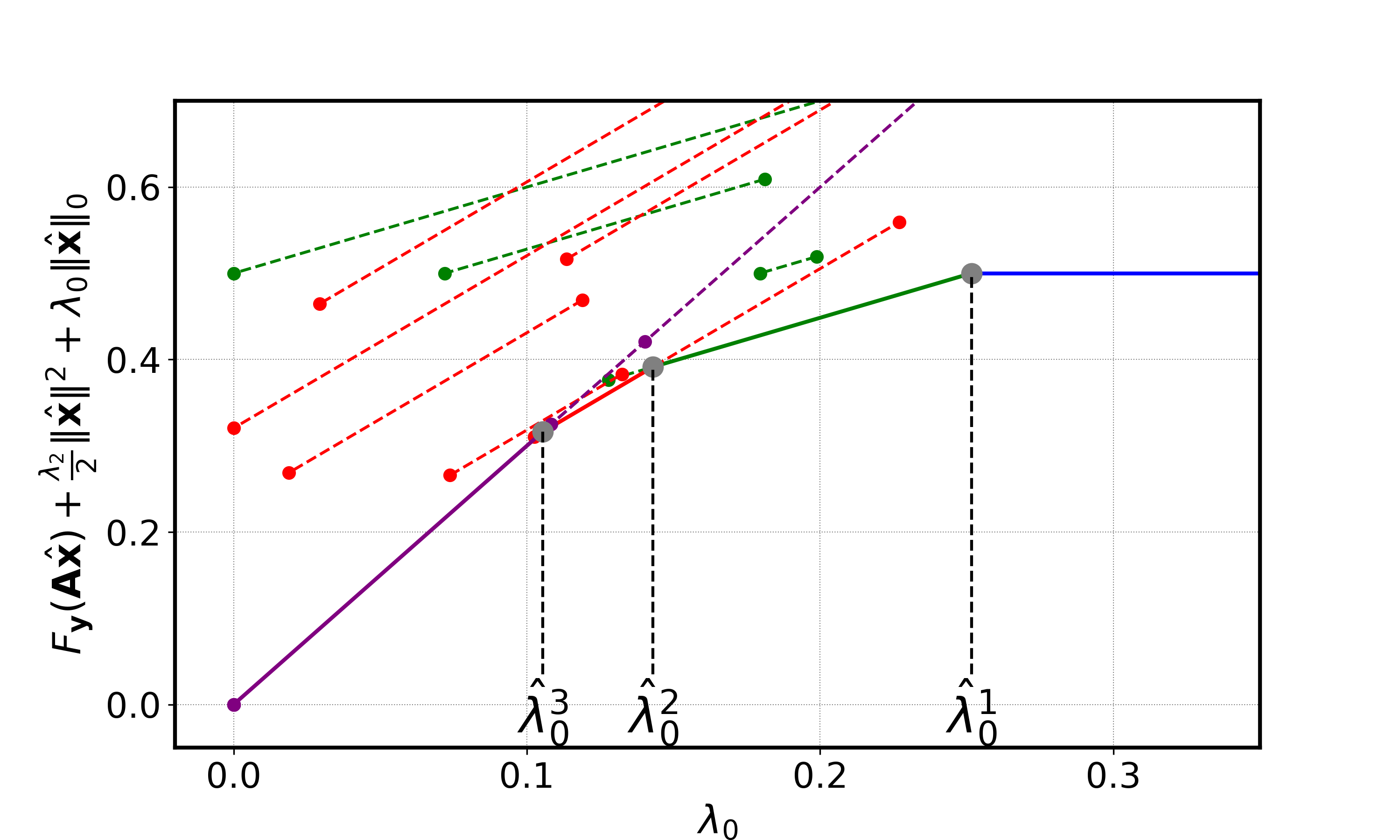}
    \caption{The $\ell_0$-curve computed using \BR{}. Plots of the segments corresponding to strict local minimizers of $J_0$ preserved by the exact relaxation~$J_\Psi$. Each segment is supported on an interval  $[\underline{\lambda}_0(\hat{\x}), \bar{\lambda}_0(\hat{\x})]$ given by Theorem~\ref{th:min-loc-removed}. This graph is the counterpart of Figure~\ref{fig:l0curve-original} when using \BR{}.
    }
    \label{fig:l0curve-relaxation}
\end{figure}

From Theorem~\ref{th:min-loc-removed}, we get that the local minimizers of the relaxation $J_\Psi$ do not correspond to lines in the $(\lambda_0, J_0)$-plane,\footnote{Equivalently, in the $(\lambda_0, J_\Psi)$-plane here, since $J_\Psi$ is an exact relaxation of $J_0$.} but rather to \emph{segments}, as illustrated in Figure~\ref{fig:l0curve-relaxation}. This property lies at the heart of the proposed \Algo{}.

\begin{remark}\label{rem:comment_ass_unif_param_psi}
    The simple expressions for the bounds  in Theorem~\ref{th:min-loc-removed} are due to Assumption~\ref{ass:unif_param_psi}. Without this assumption, the scalar functions $\psi_n$, for $n \in [N]$ could depend on $\lambda_0$, so that $\psi_n = \psi_n^{\lambda_0}$, through their parametrization when enforcing condition~\eqref{eq:cc-condition}. In such a case, before minimizing over $\sigma(\hat{\mathbf{x}})$ in~\eqref{eq:lamBar}, one would first need to isolate $\lambda_0$ in the inequality $d_{\psi_n^{\lambda_0}}(0, |\hat{x}_n|) > \lambda_0$ (and similarly for the lower bound).
\end{remark}

\subsection{Description of \Algo{}}\label{sec:proposed_algo}

Let $\mathcal{A}$ denote an algorithm guaranteed to converge to a \emph{local minimizer} of $J_\Psi$. In what follows, we denote by $\mathcal{A}(\x_0, \lambda_0)$ the execution of $\mathcal{A}$ initialized at $\x_0$ to minimize $J_\Psi$ for a fixed regularization parameter $\lambda_0$, omitting for simplicity the explicit dependence on the parameters defining the relaxation.

The core idea behind the proposed \Algo{}  is to leverage the regularization parameter range established in Theorem~\ref{th:min-loc-removed} in order to implement warm-start strategies for $\mathcal{A}$ within a tree-search framework, thereby approximating eventually the $\ell_0$-path. Different instances of \Algo{} can be deployed using various inner algorithms $\mathcal{A}$ (see Section~\ref{sec:expe}), and we will denote these  instances as \AlgoCR{$\mathcal{A}$}, with $\mathcal{A}$ specifying the name of the inner algorithm.

\subsubsection{Convergence of $\mathcal{A}$ to a Local Minimizer}

\begin{algorithm}[t]
\begin{algorithmic}[1]
\State \textbf{Require:} $\x_0 \in \R^N$, $\lambda_0 \geq 0$, Algorithm $\mathcal{A}$
\State Define $\sigma^- : \x \mapsto \{n \in \sigma(\x) : |x_n| \leq \alpha_n\}$
\State $\x = \x_0$
\While{$\sigma^-({\x}) \neq \emptyset$}
\State Pick $n \in \sigma^-(\x)$
\State $\x = \mathcal{A}(\x^{(n)}, \lambda_0)$
\EndWhile
\State \textbf{Return:} Local minimizer $\x$ of $J_\Psi$.
\end{algorithmic}
\caption{Macro algorithm to ensure the convergence to a local minimizer}
\label{algo:macro_algo}
\end{algorithm}

Algorithm $\mathcal{A}$ minimizing continuous non-convex functions such as $J_\Psi$ usually come with convergence guarantees only toward critical points, not local minimizers. However, by exploiting the properties of \BR{}, we can embed any such algorithm within an outer loop to ensure convergence to local minimizers. This approach was first introduced for the CEL0 relaxation (i.e., least-squares data fidelity and quadratic $\psi_n$) in~\citep[Section~5.1]{Soubies2015} as a ``macro algorithm'' and its analog was briefly mentioned in~\citep[Remark 4]{essafri2024exact} for the general class of \BR{}.

Let $\hat{\x}$ be a critical point of $J_\Psi$ that is not a local minimizer. By Propositions~\ref{prop:crit-pts-jpsi} and~\ref{prop:loc-min-jpsi}, there must exist at least one index $n \in \sigma(\hat{\x})$ for which $|\hat{x}_n| \leq \alpha_n$. Assuming without loss of generality that $\hat{x}_n>0$ and setting $\upsilon(t) = J_\Psi(\hat{\x}^{(n)} + t \e_n)$, we have that
$$
\begin{cases}
     \upsilon'(\hat{x}_n) = 0 & \text{(as $\hat{\x}$ is a critical point)} \\
     \upsilon \text{ strictly concave on }  (0,\alpha_n)& \text{(from~\eqref{eq:cc-condition})}
\end{cases}
$$
Given that a strictly concave function lies below its tangents, it follows that 
$$
J_\Psi (\hat{\x}^{(n)}) = \upsilon(0) < \upsilon(\hat{x}_n) = J_\Psi(\hat{\x}).
$$
Hence, $\hat{\x}^{(n)}$ can be used as an initial point for algorithm $\mathcal{A}$ to further decrease the objective function and converge to a different critical point. This process can be repeated iteratively until a critical point that is also a local minimizer is reached (see Algorithm~\ref{algo:macro_algo}), i.e., one satisfying $\forall n \in \sigma(\hat{\x}), \, |\hat{x}_n| > \alpha_n$.

Convergence of this macro algorithm in a finite number of iterations for the CEL0 relaxation was established in~\citep[Theorem~5.1]{Soubies2015}. We argue that this result naturally extends to the broader class of \BR{}, omitting the proof for brevity as it relies on similar arguments. Moreover, it is worth noting that such critical points that are not local minimizers are very unstable and algorithms generally avoid them during optimization.

\subsubsection{Details of the Algorithm}

We can now provide details on the proposed \AlgoCR{$\mathcal{A}$} algorithm, whose pseudo-code is presented in Algorithm~\ref{algo:proposed}, and its behavior is illustrated through a toy example in Figure~\ref{fig:algo-geometric}. The algorithm relies on the fact that, from Theorem~\ref{th:min-loc-removed}, a local minimizer $\hat \x$ of $J_\Psi$ obtained by $\mathcal{A}$ for $\lambda_0$ constitutes a good initial point for $\mathcal{A}$ to minimize $J_\Psi$ when the regularization parameter is chosen as
\begin{equation}\label{eq:update_lamb}
    \lambda_0 = \rho \cdot \underline{\lambda}_0(\hat{\x}) \quad \text{or} \quad \lambda_0 = \rho^{-1} \cdot \bar{\lambda}_0(\hat{\x}),
\end{equation}
for $\rho \in (0,1)$. Indeed, when $\rho$ is chosen close to $1$, these updated values of $\lambda_0$ ensure that $\hat{\x}$ is no longer a local minimizer of $J_\Psi$, while likely maintaining a relatively low objective function value (see Figure~\ref{fig:algo-geometric}, purple crosses, for an illustration).

\begin{algorithm}[t]
\begin{algorithmic}[1]
\State \textbf{Require:} $k^{\mathrm{max}} \in \N$, $N^{\mathrm{pass}} \in \N$, $\rho \in (0,1)$, Algorithm  $\mathcal{A}$
\State \textbf{Initialize:}
\State $\mathcal{X} = \{\mathbf{0} \}$ \Comment{Set of all computed points}
\State $\mathcal{X}^\mathrm{fwd} = \{\mathbf{0} \}$ \Comment{Set of points to be explored in the forward pass.}
\State $\mathcal{X}^\mathrm{bwd} = \emptyset$ \Comment{Set of points to be explored in the backward pass.}
\For{$p = 1,\ldots,N^{\mathrm{pass}}$} \Comment{Loop over passes}
    \State \textit{-- Forward Pass}
    \For{$k = 0,\ldots,k^{\mathrm{max}}$}
        \State $\x_0 = \argmin_{\x \in \mathcal{X}^\mathrm{fwd}, \, \|\x\|_0 = k } F_\y(\A\x) + \frac{\lambda_2}{2}\|\x\|^2$ \Comment{Select the point to explore}       
            \If{$\x_0 \neq \emptyset$}
            \State $\hat{\x}  \leftarrow  \mathcal{A}(\x_0, \rho \, \underline{\lambda}_0(\x_0))$ \Comment{Compute new point} \label{algo:line:newpts1}
            \If{$\hat{\x} \notin \mathcal{X}$ and $\|\hat{\x}\|_0 \leq k^{\mathrm{max}}$}
                \State $\mathcal{X}^\mathrm{fwd} = \mathcal{X}^\mathrm{fwd} \cup \{ \hat \x\}$ \Comment{Update $\mathcal{X}^\mathrm{fwd}$} 
                \State $\mathcal{X}^\mathrm{bwd} = \mathcal{X}^\mathrm{bwd} \cup \{ \hat \x\}$\Comment{Update $\mathcal{X}^\mathrm{bwd}$} 
                \State $\mathcal{X} = \mathcal{X} \cup \{ \hat \x\}$\Comment{Update $\mathcal{X}$}
                \EndIf
                \State $\mathcal{X}^\mathrm{fwd} = \mathcal{X}^\mathrm{fwd} \backslash \{ \x_0\}$ \Comment{Remove $\x_0$ from points to explore in fwd pass}
            \EndIf
    \EndFor
    \State \textit{-- Backward Pass}
    \For{$k = k^{\mathrm{max}},\ldots,0$}
        \State $\x_0 = \argmin_{\x \in \mathcal{X}^\mathrm{bwd}, \, \|\x\|_0 = k } F_\y(\A\x) + \frac{\lambda_2}{2}\|\x\|^2$ \Comment{Select the point to explore}       
            \If{$\x_0 \neq \emptyset$}
            \State $\hat{\x}  \leftarrow  \mathcal{A}(\x_0, \rho^{-1} \, \overline{\lambda}_0(\x_0))$ \Comment{Compute new point} \label{algo:line:newpts2}
             \If{$\hat{\x} \notin \mathcal{X}$ and $\|\hat{\x}\|_0 \leq k^{\mathrm{max}}$}
            \State $\mathcal{X}^\mathrm{bwd} = \mathcal{X}^\mathrm{bwd} \cup \{ \hat \x\}$\Comment{Update $\mathcal{X}^\mathrm{bwd}$} 
            \State $\mathcal{X}^\mathrm{fwd} = \mathcal{X}^\mathrm{fwd} \cup \{ \hat \x\}$\Comment{Update $\mathcal{X}^\mathrm{fwd}$}
            \State $\mathcal{X} = \mathcal{X} \cup \{ \hat \x\}$\Comment{Update $\mathcal{X}$}
            \EndIf
            \State $\mathcal{X}^\mathrm{bwd} = \mathcal{X}^\mathrm{bwd} \backslash \{ \x_0\}$  \Comment{Remove $\x_0$ from points to explore in bwd pass}
            \EndIf
    \EndFor
\EndFor
\State \textbf{Return:}  \texttt{ExtractPath}($\mathcal{X}$)
\end{algorithmic}
\caption{\AlgoCR{$\mathcal{A}$}} 
\label{algo:proposed}
\end{algorithm}

\begin{figure}[t!]
\includegraphics[width=\linewidth]{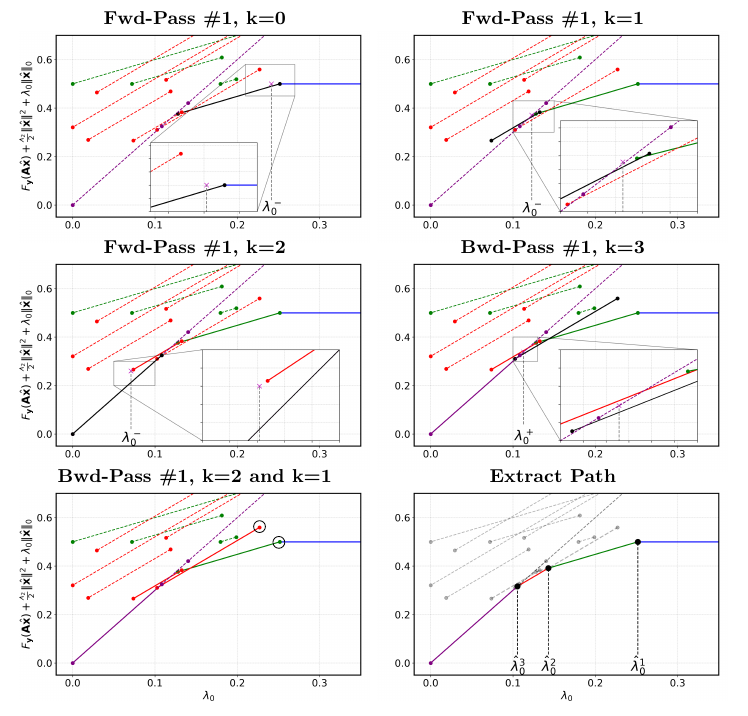}
\vspace{-0.75cm}
 \caption{\AlgoCR{$\mathcal{A}$} applied to the $\ell_0$-regularized least-squares problem illustrated in Figures~\ref{fig:l0curve-original} and~\ref{fig:l0curve-relaxation}. Here, a single forward and backward pass is performed ($N^{\mathrm{pass}} = 1$), with $k^{\mathrm{max}} = 3$ and $\rho = 0.98$. Each subplot corresponds to a step of the algorithm. In the four top plots, solid colored segments represent the current points in $\mathcal{X}$, while the purple cross marks the initially selected point $\x_0$ at $\lambda_0 = \rho  \underline{\lambda}_0(\x_0)$ for the forward pass (resp., $\lambda_0 = \rho^{-1}  \bar{\lambda}_0(\x_0)$ for the backward pass). The solid black segment denotes the newly computed point $\hat{\x}$. The bottom-left plot summarizes the steps for \(k = 2\) and \(k = 1\) during the backward pass, where the two circles highlight the explored points and their corresponding \(\lambda_0\) bounds. In both steps, the computed points $\hat{\x}$ are already present in \(\mathcal{X}\). Finally, the bottom-right plot illustrates the extraction of the final \(\ell_0\)-path from the points in \(\mathcal{X}\).
 }
\label{fig:algo-geometric}
\end{figure}

More precisely, \AlgoCR{$\mathcal{A}$}  constructs a set \(\mathcal{X}\) of points (with cardinality at most $k^{\max} \in [N]$) that are candidate to be in the \(\ell_0\)-path. Initialized as the zero vector, so that \(\mathcal{X} = \{\mathbf{0}\}\), which corresponds to the global minimizer for sufficiently large \(\lambda_0\), \AlgoCR{$\mathcal{A}$}  iteratively performs forward and backward passes that explore points in \(\mathcal{X}\), using~\eqref{eq:update_lamb} and warm-start strategies, to generate new candidate points that are sequentially added to $\mathcal{X}$. During a forward (resp., backward) pass,  points \(\x_0 \in \mathcal{X}\) are explored once, serving as  initial points for \(\mathcal{A}\) to minimize \(J_\Psi\) for values of \(\lambda_0\) slightly lower (resp., larger) than the lower bound \(\underline{\lambda}_0({\x_0})\) (resp., upper bound \(\bar{\lambda}_0({\x_0}))\). The exploration order follows increasing (resp., decreasing) cardinality among points not yet explored in that direction (loops over k). If multiple unexplored points share the same cardinality, only the one with the lowest value of \(F_\y(\A \cdot) + \frac{\lambda_2}{2}\|\cdot\|^2\)  is selected for exploration in the current pass. Finally, once the required number of passes is completed or all points in \(\mathcal{X}\) have been explored in both forward and backward directions, an estimate of the \(\ell_0\)-path is extracted from \(\mathcal{X}\), as detailed in the following paragraph. 

\paragraph{Extraction of the estimated path}
Given the final set \(\mathcal{X}\) of candidate solutions, let $\mathcal{X}_{\min} \subset \mathcal{X}$ denote the subset containing, for each cardinality $k \leq k^{\max}$, the point \(\x \in \mathcal{X}\) with $\|\x\|_0 =k$ and the lowest value of \(F_\y(\A \cdot) + \frac{\lambda_2}{2}\|\cdot\|^2\).
For two solutions \(\x_k\) and \(\x_s\) of $\mathcal{X}_{\min}$ with \(s = \|\x_s\|_0 > \|\x_k\|_0 = k\), let define
\begin{equation}
{\lambda}_0^{k,s} := \frac{1}{s - k} \max\left(F_\y(\A\x_k) + \frac{\lambda_2}{2}\|\x_k\|^2 - F_\y(\A\x_s) - \frac{\lambda_2}{2}\|\x_s\|^2, 0\right).
\end{equation}
One can then observe that, for all \(\lambda_0 \geq {\lambda}_0^{k,s}\), we have \(J_0(\x_k) \leq J_0(\x_s)\), and that the reverse inequality holds for \(\lambda_0 \leq {\lambda}_0^{k,s}\).
We then define the critical values associated with \(\x_k\) as
\begin{equation}\label{eq:lam0-critic}
\underline{\lambda}_0^k = \max_{s \in [k+1, k^{\mathrm{kmax}}]} {\lambda}_0^{k,s}, \quad \text{and} \quad \bar{\lambda}_0^k = \min_{s \in [0,k-1]} {\lambda}_0^{s,k},
\end{equation}
so that, for all \(\lambda_0 \in \Lambda^k := (\underline{\lambda}_0^k, \bar{\lambda}_0^k)\), the solution \(\x_k\) dominates all other candidate solutions in $\mathcal{X}_{\min}$. The estimated \(\ell_0\)-path is then formed by the points \(\x_k \in \mathcal{X}_{\min}\) for which the interval \(\Lambda^k\) is non-empty, and the associated $(\Lambda^k)_k$ form a partition of $\R_{\geq 0}$ (see last plot of Figure~\ref{fig:algo-geometric}).

\subsubsection{Boosting \Algo{}}\label{sec:local-search}

Conceptually, the proposed \Algo{} shares similarities with \texttt{L0Learn}~\citep{CWLSHazimeh:20,HazimehClass}.
The key distinctions lie the facts that  \Algo{} exploits exact relaxations $J_\Psi$ and minimizes them with any suitable algorithm, while performing a forward-backward search over $\lambda_0$. In contrast, \texttt{L0Learn} directly addresses $J_0$ using a coordinate descent (CD) method with a single forward pass over decreasing values of $\lambda_0$. In this section, we describe three ``boosting'' strategies, originally proposed in \texttt{L0Learn} to enhance both computational efficiency and  quality of candidate solutions to the $\ell_0$-path, that we also adopt in \Algo{}.

\paragraph{Local combinatorial search.}

By construction, coordinate descent (CD) converges to coordinate-wise (CW) minima, i.e., points for which the objective function cannot be decreased by varying a single coordinate. A stronger optimality condition, introduced by the authors of \texttt{L0Learn}, is based on the notion of \emph{partial swap inescapable} (PSI) minima~\citep[Definition 3]{CWLSHazimeh:20}: a point is a PSI minimum if the objective function cannot be improved by swapping any support component with any off-support component, followed by a ``projection'' onto the new support. As such, given a CW minimizer $\hat{\x}$ obtained with CD, they propose  to refine it by constructing a candidate point $\tilde{\x} = \hat{\x}^{(i)} + \tilde{z}\mathbf{e}_j$, obtained by swapping $i \in \sigma(\hat{\x})$, $j \in \sigma^c(\hat{\x})$ with 
\begin{equation}\label{eq:local-search}
\tilde{z} \in  \argmin_{z \in \Cc} \; 
F_\y(\A\hat\x^{(i)} + z \mathbf{a}_j) + \frac{\lambda_2}{2}z^2,
\end{equation}
satisfying $F_\y(\A\Tilde{\x}) + \frac{\lambda_2}{2}\|\Tilde{\x}\|^2 < F_\y(\A{\x}) + \frac{\lambda_2}{2}\|{\x}\|^2$ (if such one exists, i.e., $\hat{\x}$ is not PSI). For quadratic fidelities $F_\y(\A\cdot)$,  an efficient algorithm of linear complexity (with careful implementation) exists to find  such a point when it exists. It  relies on a closed-form solution for~\eqref{eq:local-search}~\citep{CWLSHazimeh:20}.  For non-quadratic fidelities with Lipschitz gradients, $F_\y(\A\cdot)$ can be upper-bounded by a quadratic function, enabling the use of the same closed-form solution within a majorization-minimization framework~\citep{HazimehClass}. Finally, in \texttt{L0Learn} with \texttt{PSI} local search, CD is run a second time from $\tilde{\x}$.

We also implement this \texttt{PSI} strategy in \Algo{} and describe below the two key components required to efficiently compute $\tilde{\x}$ for quadratic fidelities. Since this strategy relies on the CW stationarity of the input $\hat{\x}$, a condition not always guaranteed when using $J_\Psi$ and off-the-shelf minimization algorithms to minimize it, some adaptations to the method proposed in~\citep{CWLSHazimeh:20} are necessary. Let $\mathbf{r} := \A\hat{\x} - \y$, then after some computations one gets that 
$$
\tilde{z} = \frac{\a_j^T(\a_i \hat{x}_i - \mathbf{r})}{\|\a_j\|^2 + \lambda_2},
$$
and that
\begin{align*}
 & F_\y(\A\Tilde{\x}) + \frac{\lambda_2}{2}\|\Tilde{\x}\|^2 < F_\y(\A{\x}) + \frac{\lambda_2}{2}\|{\x}\|^2 \\  
 \Longleftrightarrow & (\|\a_j\|^2 + \lambda_2) z^2 > (\|\a_j\|^2 - \lambda_2) \hat{x}_i^2 - 2(\a_i^T \mathbf{r})\hat{x}_i .
\end{align*}
The term $2(\a_i^T \mathbf{r})\hat{x}_i$ is not present in~\citep{CWLSHazimeh:20} as, for $\hat{\x}$ CW minimizer of $J_0$, $\a_i^T \mathbf{r} = 0$ for all $i \in \sigma(\hat{\x})$.
Since most solvers already compute quantities such as $\a_j^T \a_i$ and $\a_j^T \mathbf{r}$, the candidate point $\tilde{\x}$ defined above can be efficiently determined using the two previous expressions.

\paragraph{Correlation screening.}

Unsurprisingly, and as confirmed by our experimental results, increasing the number of passes in \Algo{} improves the quality of the estimated $\ell_0$-path. However, this also increases the number of points explored and the number of calls to the inner solver $\mathcal{A}$, thus affecting computational efficiency. Therefore, improving the efficiency of this inner algorithm is critical. A natural approach to accelerate the inner solver is to reduce the size of the problem. Following~\citep{CWLSHazimeh:20}, we restrict the execution of $\mathcal{A}$ only on the support of $\x_0$, augmented by (at most) $N^\mathrm{screen}$ coordinates that are the most likely to contribute to the solution. All coordinates outside the following set $I$ are set to zero
\[
I = \sigma(\x_0) \cup \left| \nabla G(\x_0) \right|^{\downarrow}_{1:N^\mathrm{screen}},
\]
where $\left| \cdot \right|^{\downarrow}$ denotes the sorting operator in decreasing order of magnitude, and $N^\mathrm{screen} < N$ is a parameter controlling the screening size.

\paragraph{Optimization over stabilized support} During the iterations of many algorithms $\mathcal{A}$, the support often stabilizes before the on-support coordinates have fully converged. Thus, when the support $\sigma^{\mathrm{stab}} := \sigma(\x^l)$ remains unchanged for multiple consecutive iterations $l$ (typically of the order of 10 in practice), we terminate the algorithm and finalize convergence on the support by solving the restricted convex and smooth problem
\[
\hat{\x}_{\sigma^{\mathrm{stab}}} \in \argmin_{\mathbf{z} \in \Cc^{\sharp \sigma^{\mathrm{stab}}}} F_\y(\A_{\sigma^{\mathrm{stab}}} \mathbf{z}) + \frac{\lambda_2}{2}\|\mathbf{z}\|^2.
\]
This ensures that the final solution is obtained by refining only the active coordinates on the stabilized support. We solve this problem using an accelerated gradient descent with backtracking line-search, but of course faster method, such as, e.g., Newton-type approaches could be used.

\section{Experiments}\label{sec:expe}

In this section, we evaluate the performance of our proposed \Algo{} method, focusing on both optimization quality and statistical performance. We compare it with state-of-the-art sparse optimization methods on synthetic and real-world datasets, considering both sparse least-squares and logistic regression problems. {To analyse the influence of the choice of generating functions $\Psi$, we further consider synthetic sparse Kullback-Leibler regression problems. }

\subsection{Problems, Data,  Baselines, and Metrics}\label{sec:problems_metrics}

\subsubsection{Regression Problems}

We focus on the following three instances of Problem~\eqref{eq:problem_setting}.

\paragraph{Sparse least-squares regression (LS).} It corresponds to Problem~\eqref{eq:problem_setting}  with $\A \in \R^{M \times N}$, $\y \in \R^M$, $\Cc = \R$, $\lambda_2 = 0$, and the loss function given by:
    \begin{equation}\label{eq:ls-data}
        F_{\y}^{\mathrm{LS}}(\z)
        = \frac{1}{2M}\| \z - \y\|^2.
    \end{equation}  
This setting arises, e.g., in linear regression and signal/image reconstruction problems with Gaussian noise.
\paragraph{Sparse logistic regression (LR).} It corresponds to Problem~\eqref{eq:problem_setting}  with $\A \in \R^{M \times N}$, $\y \in \{-1,1\}^M$, $\Cc = \R$, $\lambda_2 > 0$,  and the loss function given by:
    \begin{equation}
        F_\y^{\mathrm{LR}}(\z)
        = \frac{1}{M}\sum_{m=1}^M \log\!\left( 1 + \exp\!\left(-y_m z_m \right) \right).
    \end{equation}
    This setting arises in binary classification with labels  $\y \in \{-1,1\}^M$. 

\paragraph{Sparse Kullback-Leibler regression (KL).}
{It corresponds to  Problem~\eqref{eq:problem_setting}  with $\A \in \R_{\geq 0}^{M \times N}$, $\y \in \R_{\geq 0}^M$, $\Cc = \R_{\geq 0}$, $\lambda_2 = 0$ and the loss function given by:
    \begin{equation}
        F_\y^{\mathrm{KL}}(\z)
        = \frac{1}{M}\sum_{m=1}^M d_{\mathrm{KL}}(y_m, z_m + b_m)
    \end{equation}
    with $\mathbf{b} \in \R^M_{>0}$ and $d_{\mathrm{KL}}$ the one-dimensional KL divergence defined in Example~\ref{ex:PsiKL}. Such data term arises in the special case of sparse Poisson regression in which data are counts, i.e., $\y \in \mathbb{Z}_{\geq 0}^M$, drawn from a Poisson distribution, as well as in signal/image reconstruction problems under photon-counting noise.
}

In all three cases, we normalize the columns of $\A$. 
We also note that we consider $\lambda_2 = 0$ in both LS and KL settings, since our goal is to focus on the pure $\ell_0$-regularized problem. However, for LR, a positive value of $\lambda_2$ is required to ensure the well-posedness of the optimization problem \citep[Theorem 1]{essafri2024exact}. In particular, Problem~\eqref{eq:problem_setting} admits a solution whenever either $F_{\y}$ is coercive or $\lambda_2 > 0$. Since coercivity fails in the logistic regression setting, we therefore consider $\lambda_2 > 0$.

\subsubsection{Synthetic Data Generation}\label{sec:data_generation}

In our experiments with simulated data, we generate a measurements matrix $\A \in \mathbb{R}^{M \times N}$ sampled from a multivariate normal distribution with zero mean and covariance matrix $\bm{\Sigma} = \left(\Sigma_{ij}\right)_{1 \leq i,j \leq N} \in \mathbb{R}^{N \times N}$ satisfying one of the following correlation structures:
\begin{itemize}
    \itemsep0pt
    \item \emph{Constant correlation:} $\Sigma_{ij} = \eta$ for all $i \neq j$, and $\Sigma_{ii} = 1$ for all $i \in [N]$.
    \item \emph{Exponential correlation:} $\Sigma_{ij} = \eta^{|i-j|}$ for all $i, j$, with $\Sigma_{ii} = 1$.
\end{itemize}
In both cases, the parameter $\eta \in (0,1)$ controls the level of correlation between features. {In addition, for KL problems,  we ensure the nonnegativity of $\A$ by retaining only the absolute values of its entries.} We now describe the generation of the observation vector $\y$. First, we construct a sparse ground truth vector $\x^* \in \mathbb{R}^N$ with $k$ evenly-spaced nonzero entries of unit amplitude in absolute value.

\paragraph{Least-squares regression.}  
The observation vector is generated as
$
\y = \A\x^* + \bm{\upvarepsilon},
$
where $\bm{\upvarepsilon} \sim \mathcal{N}(0, \sigma^2 \mathbf{I})$ is a Gaussian vector of noise. The noise variance $\sigma^2$ is chosen to match a desired signal-to-noise ratio (SNR), defined as
$
\mathrm{SNR} = {\operatorname{Var}(\A\x^*)} / {\operatorname{Var}(\bm{\upvarepsilon})}.
$

\paragraph{Logistic regression.}  
The binary response vector $\y \in \{-1,1\}^M$ is generated according to a Bernoulli distribution. For each $m \in  [M]$, the probability of the label being $1$ is 
\[
P(y_m = 1 \mid \underline{\a}_m) = \left(1 + \exp\left(-s \langle \underline{\a}_m, \x^* \rangle\right)\right)^{-1},
\]
where $\underline{\a}_m \in \R^N$ denotes the $m$th row of $\A$ and
 $s > 0$ controls the signal-to-noise ratio. Smaller values of $s$ produce probabilities closer to $1/2$, leading to noisier labels, while larger values increase class separability. In the limit $s \to \infty$, the data become linearly separable. This model is commonly used to simulate binary classification problems with sparse features.

\paragraph{Kullback-Leibler regression.}  
{
We consider the special case of Poisson regression, where the counts response vector $\y \in \mathbb{Z}^M_{\geq 0}$ is generated according to $\y \sim \mathcal{P}(\varrho(\A \x^* + \mathbf{b}))/\varrho$. Here, $\mathcal{P}$ denotes the Poisson distribution, $\mathbf{b} \in \mathbb{R}_{> 0}^M$ is a known offset vector (e.g., modelling background signal in imaging~\citep{fessler1998paraboloidal}), and $\varrho >0$ is a gain factor that controls the level of Poisson noise.
}

\subsubsection{Real datasets}\label{sec:real_data}

We evaluate our proposed \Algo{} method on several real-world datasets, summarized in Table~\ref{tab:real-data}.
For the sparse least-squares problem, we consider two gene expression datasets: the \texttt{RIBOFLAVIN} dataset~\citep{ribovlafin} and the \texttt{NCI-60} dataset~\citep{ross2000systematic}.
For sparse logistic regression, we use two cancer-related classification datasets: the \texttt{COLON CANCER} dataset~\citep{Alon1999BroadPO}, which includes tumor and normal colon tissue samples, and the \texttt{LEUKEMIA} dataset~\citep{golub1999molecular}, which focuses on classifying leukemia subtypes based on gene expression data.
Additionally, we also evaluate the least-squares formulation on the two classification datasets, following common practice in the literature.

\begin{table}[t]
\centering
\begin{tabular}{|l|c|c|c|}
\hline
\cellcolor[gray]{0.9} {\textbf{Dataset}}      &\cellcolor[gray]{0.9}  \textbf{$F_\y$} & \cellcolor[gray]{0.9} \textbf{M} & \cellcolor[gray]{0.9} \textbf{N}    \\ \hline
\texttt{RIBOFLAVIN}~\citep{ribovlafin}            & Least-squares          & 71          & 4088          \\ 
\texttt{NCI-60}~\citep{ross2000systematic}                & Least-squares             & 64          & 6830           \\ 
\texttt{COLON-CANCER}~\citep{Alon1999BroadPO}         & Least-squares/Logistic                & 62          & 2000          \\ 
\texttt{LEUKEMIA}~\citep{golub1999molecular}             & Least-squares/Logistic               & 38          & 7129          \\ \hline
\end{tabular}
\caption{Real-world datasets considered in this work, along with their dimensions and regression problems on which they are used.}
\label{tab:real-data}
\end{table}

\subsubsection{\Algo{} Parameter Settings}\label{sec:base_algos} 

To deploy \Algo{}, we first define an exact relaxation $J_\Psi$ of $J_0$. For both LS and LR problems, we  set $\psi_n(x) = \frac{\gamma_n}{2}x^2$. For LS, this choice nicely aligns with the geometry of the data fidelity term. While this is less the case for LR, it is a practical choice among generating functions that allow \(\beta_\psi\) to be computed in closed form~\citep{essafri2024exact}. 

According to Example~\ref{ex:PsiL2}, and given that we normalize the columns of $\A$, we choose $\gamma_n = \frac{1}{M}$ for LS and $\gamma_n = \frac{1}{4M} + \lambda_2$ for LR  so as to ensure the exact relaxation property. Note that in both cases Assumption~\ref{ass:unif_param_psi} is satisfied as shown in Example~\ref{ex:PsiL2}. Finally, in this context, the bounds in Theorem~\ref{th:min-loc-removed} are given by
\begin{equation}\label{eq:lamb_bounds_L2}
\bar{\lambda}_0(\hat \x) = \min_{n \in \sigma(\hat{\x})} \, \frac{\gamma_n}{2}\hat{x}_n^2 \; \text{ and } \; \underline{\lambda}_0(\hat \x) = \max_{n \in \sigma^c(\hat{\x})} \, \frac{\left(\varphi_n\left(- \left<\a_n ,\nabla F_\y\left(\A\hat{\x} \right) \right>\right)\right)^2}{2 \gamma_n}    
\end{equation}
with $\varphi_n = |\cdot|$.

{Regarding KL problems,  we consider two choices of generating functions: $\psi_n(x) = \frac{\gamma_n}{2}x^2$ (referred to as $\Psi_{\ell_2}$) and $\psi_n(x) = \gamma_n d_{\mathrm{KL}}(\xi ; c_n x + \xi)$ (referred to as $\Psi_{\mathrm{KL}}$). To ensure the exact relaxation property, we set $\gamma_n = \sum_{m \in [M]} {a_{mn}^2 y_m}/{(b_m^2 M)}$ for $\Psi_{\ell_2}$, and $(\gamma_n, c_n, \xi)$ as described in Example~\ref{ex:PsiKL} for $\Psi_{\mathrm{KL}}$. In both cases, Assumption~\ref{ass:unif_param_psi} is satisfied. While for $\Psi_{\ell_2}$ the bounds in Theorem~\ref{th:min-loc-removed} are given by~\eqref{eq:lamb_bounds_L2} (with $\varphi_n = \max(\cdot,0)$), for $\Psi_{\mathrm{KL}}$ they can be simplified as
\begin{equation}
\begin{split}
    & \bar{\lambda}_0(\hat \x) = \min_{n \in \sigma(\hat{\x})}  \gamma_n \xi \left( \log\left( 1 + \frac{c_n}{\xi}|\hat{x}_n| \right) - \frac{c_n |\hat{x}_n|}{c_n |\hat{x}_n| + \xi} \right) \\
    & \underline{\lambda}_0(\hat \x) = \max_{n \in \sigma^c(\hat{\x})} \,  - \gamma_n \xi \left( \log\left( 1 - \frac{\varphi_n(g_n)}{\gamma_n c_n}\right)  + \frac{\varphi_n(g_n)}{\gamma_n c_n} \right)
\end{split} 
\end{equation}
with  $\varphi_n = \max(\cdot,0)$ and $g_n = - \left<\a_n ,\nabla F_\y\left(\A\hat{\x} \right) \right>$.
}

Next, we select an inner algorithm ($\mathcal{A}$) designed to minimize the exact relaxation $J_\Psi$.
In our experiments, we consider the following three algorithms:
\begin{itemize}
\itemsep0pt
    \item \texttt{FBS} (Forward-Backward Splitting~\citep{combettes2011proximal,attouch2010proximal}), also known as proximal gradient.
    \texttt{FBS} combines a gradient step on the smooth term with the proximal operator of $B_\Psi$.
    Its application to minimize $J_\Psi$ is detailed in~\citep{essafri2024exact}, including the derivation of the proximal operator of $B_\Psi$.

    \item \texttt{IRL1} (Iteratively Reweighted $\ell_1$~\citep{Ochs2015OnIR,chouzenoux2016block}).
    \texttt{IRL1} is a majorization-minimization algorithm that optimizes $J_\Psi$ by solving a sequence of convex weighted-$\ell_1$ subproblems.
    Each subproblem is obtained by majorizing the folded-concave unidimensional functions $\beta_{\psi_n}$ using their tangents.
    This requires the subdifferential of $\beta_{\psi_n}$, which is provided in~\citep{essafri2024exact}. The subproblems are minimized using an accelerated \texttt{FBS}~\citep{beck2009fast,nesterov2013gradient}.

    \item \texttt{skglm}~\citep{skglm}.
    \texttt{skglm} is a working-set algorithm based on Anderson-accelerated coordinate descent.
    It has been shown to perform effectively on non-convex continuous functionals such as $J_\Psi$.
\end{itemize}
We denote the resulting algorithmic variants by \AlgoCR{\texttt{FBS}}, \AlgoCR{\texttt{IRL1}}, and \AlgoCR{\texttt{skglm}}. 

Regarding the parameters of the outer \Algo{} strategy, we set $\rho = 0.95$ in all experiments, except for the LS experiment with the \texttt{LEUKEMIA} dataset, where we use $\rho = 0.88$.
This adjustment turned to be necessary for this specific example to prevent the algorithm from terminating prematurely before reaching the given time limit.
For the heuristics presented in Section~\ref{sec:local-search}, we adopt the following settings:
\begin{itemize}
\itemsep0pt
    \item \textit{For synthetic datasets:}
    We disable the screening heuristic (i.e., $N^\mathrm{screen} = N$) and explicitly indicate when the local search \texttt{PSI} is employed.
    \item \textit{For real datasets:}
    We always activate \texttt{PSI} and set $N^\mathrm{screen} = 1000$, except for \AlgoCR{\texttt{skglm}}.
    Since \texttt{skglm} is a working-set method, it inherently avoids optimizing over the entire vector in  early iterations. Moreover, we observed that it achieves better performance without the screening heuristic.
\end{itemize}
Finally, we terminate \Algo{} either when a maximal number of passes $N^\mathrm{pass}$ is reached or when the predefined time limit is attained. 

All experiments were run on a laptop with an Intel Core i7-13800H CPU and 32 GB of RAM under Ubuntu 22.04.5 LTS (64-bit).

\subsubsection{Benchmark Algorithms}

For both LS and LR experiments, we benchmark \AlgoCR{\texttt{FBS}}, \AlgoCR{\texttt{IRL1}}, and \AlgoCR{\texttt{skglm}} against \Algolearn{}~\citep{CWLSHazimeh:20,HazimehClass}.
For LS real datasets, we also compare our methods with  \texttt{L0PD} ($\ell_0$-path descent) \citep{SoussenPath:15}, a greedy algorithm that leverages the concave structure of the $\ell_0$-curve to estimate the $\ell_0$-path. {Since neither \Algolearn{} nor \texttt{L0PD} supports KL problems, we leverage KL experiments to analyse the effect of choosing $\Psi_{\ell_2}$ or $\Psi_{\mathrm{KL}}$ to generate the exact relaxation, employing {\texttt{IRL1}} as inner solver. }

\subsubsection{Metrics}\label{sec:metrics}

To evaluate the performance of the considered methods, we use the following statistical metrics for simulated datasets, computed between an estimated vector $\hat{\x}$ and the ground-truth vector $\x^*$:
\begin{itemize}
    \itemsep0pt
    \item \emph{F1-score}: A metric measuring the accuracy of support recovery, defined as:
\[
\mathrm{F1} = \frac{2 \cdot \mathrm{TP}}{2 \cdot \mathrm{TP} + \mathrm{FN} + \mathrm{FP}},
\]
where $\mathrm{TP}$, $\mathrm{FN}$, and $\mathrm{FP}$ denote the number of true positives, false negatives, and false positives, respectively. An F1-score closer to one indicates better recovery of the support of $\x^*$ by $\hat{\x}$.
    \item \emph{Relative Root Mean Squared Error (RMSE)}: A metric measuring the estimation accuracy of $\hat{\x}$, defined as:
\[
\mathrm{RMSE} = \frac{\|\x^* - \hat{\x}\|}{\|\x^*\|}.
\]
Lower values of RMSE correspond to more accurate signal estimation.
\end{itemize}

For real datasets, where no ground-truth vector $\x^*$ is available, we evaluate the performance of the methods by examining the \emph{Pareto front} associated with the estimated $\ell_0$-path. Specifically, we plot the objective function value $F_\y(\A\hat{\x}) + \frac{\lambda_2}{2}\|\hat{\x}\|^2$ against the sparsity level $\|\hat{\x}\|_0$. A lower Pareto front indicates a better trade-off between fit error and sparsity, i.e., a lower fit error for a given sparsity level.

\begin{remark}
Since we address penalized $\ell_0$ problems, the reported ``Pareto fronts'' actually represent estimates of the {convex envelope} of the true Pareto front, as defined in the context of bi-objective optimization.
Indeed, for $\ell_0$-based bi-objective problems, the Pareto front may be  non-convex, and the penalized formulation cannot access points located in non-convex regions of the front~\citep[Figure~2]{SoussenPath:15}.
\end{remark}

\subsection{Sparse Least-Squares  Regression}\label{sec:LS}

\begin{figure}[!t]
\centering
\includegraphics[width=\linewidth]{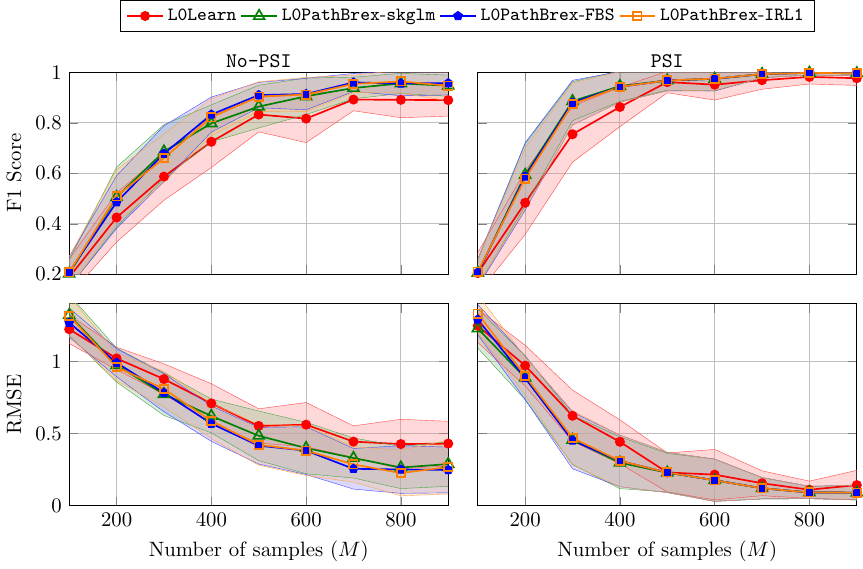}
\caption{
LS. Average F1-score (top) and RMSE (bottom) as functions of the number of samples, computed over 20 experiments. Shaded regions indicate $\pm$ one standard deviation around the average. For the \Algo{} variants, $N^{\mathrm{pass}} =4$ passes are performed. The left column corresponds to methods without local search (\texttt{No-PSI}), while the right column shows results when all methods are combined with the local search strategy \texttt{PSI}. The data are generated with parameters $(N, \mathrm{SNR}, \eta, k) = (1000, 5, 0.9, 25)$ under an exponential correlation structure of the covariance matrix used to generate $\A$. For all methods, the path is computed up to $k^{\max} = 2\cdot k = 50$.\label{fig:LS_Simu_Exp_Corr}
}
\end{figure}

In this section, we present  results for sparse least-squares regression on both synthetic and real datasets.

\paragraph{Statistical performance as a function of number of samples.}
We evaluate the statistical performance of \Algo{} and \texttt{L0Learn} using the F1-score and RMSE metrics, while varying the number of samples \(M\). The remaining parameters \((N, \mathrm{SNR}, \eta, k)\) are fixed at \((1000, 5, \eta, 25)\), where \(\eta\) is set differently depending on the correlation setting.
We report means with standard deviations computed over 20 independent problem instances.
For each estimated \(\ell_0\)-path, we select the point \(\hat{\x}\) that maximizes the F1-score and compute both metrics at this optimal point.

In Figures~\ref{fig:LS_Simu_Exp_Corr} and~\ref{fig:LS_Simu_Cst_Corr}, we report results for exponential (with \(\eta = 0.9\)) and constant (with \(\eta = 0.7\)) correlation settings, respectively. Each figure also compares the performance of the methods with and without the local search \texttt{PSI}.

In both correlation settings, the proposed \Algo{} consistently outperforms \texttt{L0Learn}, regardless of the inner solver $\mathcal{A}$.
Additionally, activating the local search \texttt{PSI} further enhances the performance of all methods.
Finally, as expected, performance is slightly degraded in the more challenging constant correlation setting.

These  gains come at the cost of an increased computational cost.
For the four passes considered in our experiments, the computational times of the \Algo{} variants (without \texttt{PSI}) range from a few seconds for the \texttt{skglm} variant to a few tens of seconds for the \texttt{IRL1} variant with the \texttt{FBS} variant in between, closer to \texttt{skglm}.
In contrast, the computational time of \texttt{L0Learn}  is of the order of 0.1 seconds. It is important to note that the computational cost of \Algo{} stems not only from the inner solver, but also from the number of points explored during the given number of passes, which differs depending on the inner solver. Furthermore, it is important to note that while all \Algo{} variants are implemented in \texttt{Python}, \texttt{L0Learn} is based on a \texttt{C++} implementation, a distinction that should be considered when interpreting these computational time comparisons.

Overall, in these experiments, \AlgoCR{\texttt{skglm}} and \AlgoCR{\texttt{FBS}} emerge as the methods offering the best compromise between statistical performance and computational efficiency.

\begin{figure}[!t]
\centering
\includegraphics[width=\linewidth]{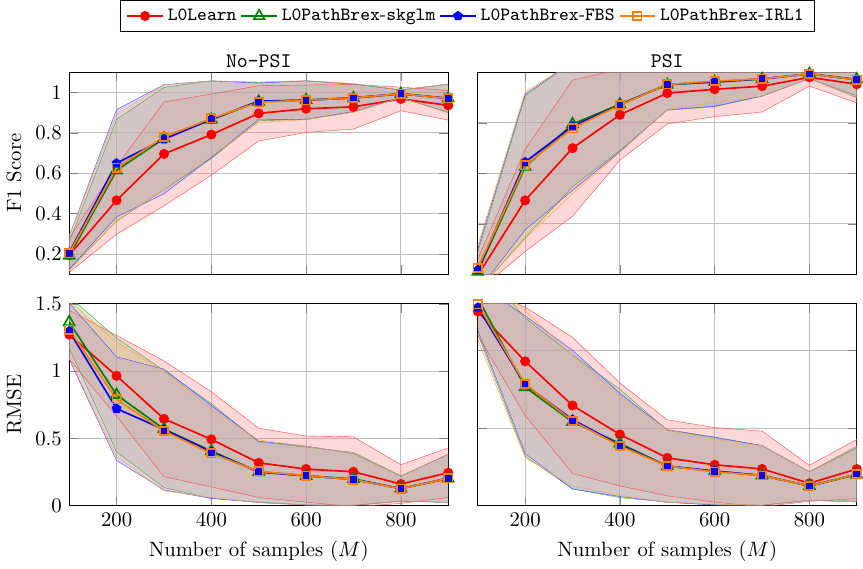}
\caption{
LS. Same as Figure~\ref{fig:LS_Simu_Exp_Corr} for the constant correlation setting with $\eta = 0.7$.\label{fig:LS_Simu_Cst_Corr}
}
\end{figure}

\paragraph{Pareto front estimation on real datasets.} Figure~\ref{fig:LS-real-60seconds} presents the estimated Pareto fronts for the real datasets of Table~\ref{tab:real-data}.
Unlike the simulated experiments, we fix a time limit of one minute for all \Algo{} variants, rather than fixing the number of passes $N^{\mathrm{pass}}$.

\begin{figure}[!t]
    \centering
    \includegraphics[width=\linewidth]{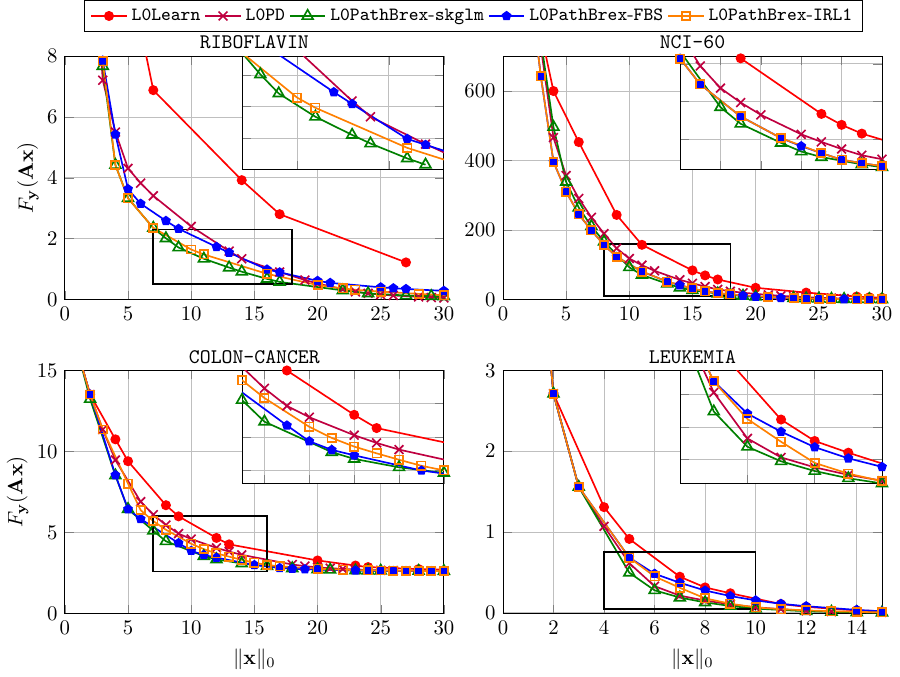}
    \vspace{-0.5cm}
    \caption{Estimated Pareto fronts for the LS problem on the four real datasets from Table~\ref{tab:real-data} with a one-minute time limit for all \Algo{} variants.}
    \label{fig:LS-real-60seconds}
\end{figure}

The results show that all \Algo{} variants consistently outperform (lower curves) \texttt{L0Learn} by a significant margin. 
They also surpass the second baseline, \texttt{L0PD}, though to a lesser extent. The only exception is the \texttt{LEUKEMIA} dataset, where only \AlgoCR{\texttt{skglm}} outperforms \texttt{L0PD}. In other words, \Algo{} finds sparser points that also exhibit better data fidelity.
Furthermore, we observe  differences among the considered inner solvers $\mathcal{A}$.
In particular, \texttt{skglm} consistently achieves the best results, while the relative performance of \texttt{IRL1} and \texttt{FBS} varies across datasets, with each outperforming the other in different cases.

We supplement these results with two additional figures in the appendix: Figures~\ref{fig:LS-real-15seconds} and~\ref{fig:LS-real-300seconds}, which present the same experiment with time limits of 15 seconds and 5 minutes, respectively. Several observations can be made depending on the dataset.
On the one hand, for the \texttt{NCI-60} and \texttt{LEUKEMIA} datasets, the estimated Pareto curves remain identical between the three reported time limits.
This indicates that, on these datasets, the estimated $\ell_0$-path stabilizes within the first 15 seconds, and any additional points computed afterward are pruned during the final path extraction. On the other hand, for the \texttt{RIBOFLAVIN} and \texttt{COLON-CANCER} datasets, the estimated Pareto curves obtained after only 15 seconds (Figure~\ref{fig:LS-real-15seconds}) are noticeably degraded compared to those obtained after one minute (Figure~\ref{fig:LS-real-60seconds}).
In particular, their performance at best reaches that of \texttt{L0PD}, yet still remains superior to that of \texttt{L0Learn}.
This suggests that, for these two datasets, extending the passes to one minute significantly enhances the quality of the estimated $\ell_0$-path.
However, when examining the results after five minutes (Figure~\ref{fig:LS-real-300seconds}), we observe two distinct behaviours.
For the \texttt{RIBOFLAVIN} dataset, the curves have further improved, whereas for the \texttt{COLON-CANCER} dataset, they remained identical to that observed after one minute.

Overall, the proposed \Algo{} approach provides a good balance between  quality of the estimated $\ell_0$-path and computational efficiency.
While \texttt{L0PD} executes in a order of one second and \texttt{L0Learn} executes in a few seconds on these examples, \Algo{} achieves notable quality improvements at a reasonable computational cost, especially given the inherent complexity of such $\ell_0$-regularized problems.

\begin{remark}
    For the \texttt{COLON-CANCER} dataset, one may observe that reported curves for all methods do not converge to zero as the support size increases.
First, note that the matrix $\A$ has full row rank $M$ in this case, implying that supports of size at least $M$ exist for which perfect data fidelity (i.e., zero error) can be achieved. We identify two potential explanations for the observed behavior:
\begin{itemize}
    \itemsep0pt
    \item All methods struggle to get high-quality solutions (i.e., with low data fidelity) for cardinalities between, say, 15 and $M$. Within this range, many local minima exhibit similar data fidelity values (resulting in the observed ``plateau''), while the rare, high-quality solutions that would drive the Pareto curve toward zero remain difficult to reach.
    \item The measurement vector $\y$ lies at a nearly constant distance from the subspaces generated by any $q \in [15,M]$ columns of $\A$.
\end{itemize}
\end{remark}

\subsection{Sparse Logistic Regression}\label{sec:LR}

In this section, we present  results for sparse logistic regression on both synthetic and real datasets.

\begin{figure}[!t]
\centering
\includegraphics[width=\linewidth]{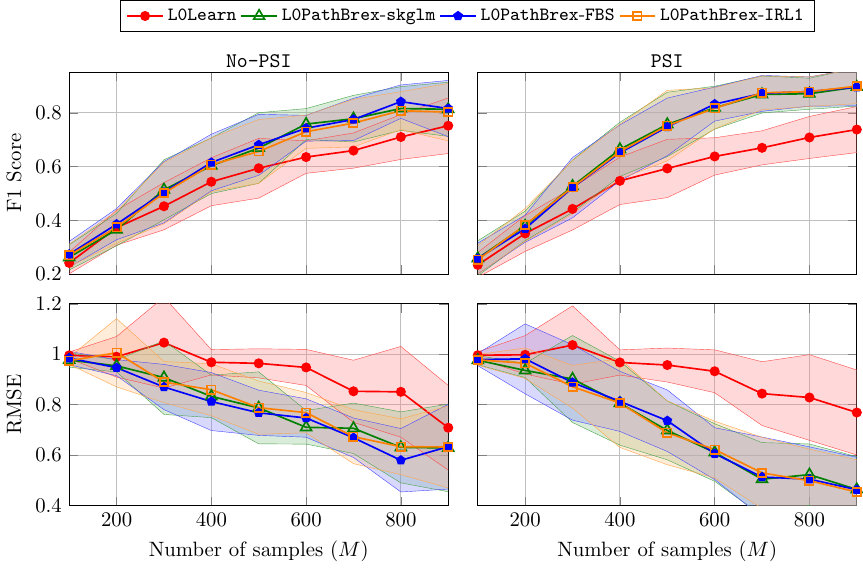}
\caption{
LR. Average F1-score (top) and RMSE (bottom) as functions of the number of samples, computed over 20 experiments. Shaded regions indicate $\pm$ one standard deviation around the average. For the \Algo{} variants, $N^{\mathrm{pass}} =4$ passes are performed. The left column corresponds to methods without local search (\texttt{No-PSI}), while the right column shows results when all methods are combined with the local search strategy \texttt{PSI}. The data are generated with parameters $(N, s, \eta, k) = (1000, 25, 0.9, 25)$ under an exponential correlation structure of the covariance matrix used to generate $\A$. For all methods, the path is computed up to $k^{\max} = 2\cdot k = 50$. Regarding the parameter $\lambda_2$, we consider a logarithmically spaced grid of 10 values in $[10^{-6}, 10]$, and select the solution achieving the highest F1-score over this grid.
\label{fig:LR_Simu_Exp_Corr}
}
\end{figure}

\begin{figure}[!t]
\centering
\includegraphics[width=\linewidth]{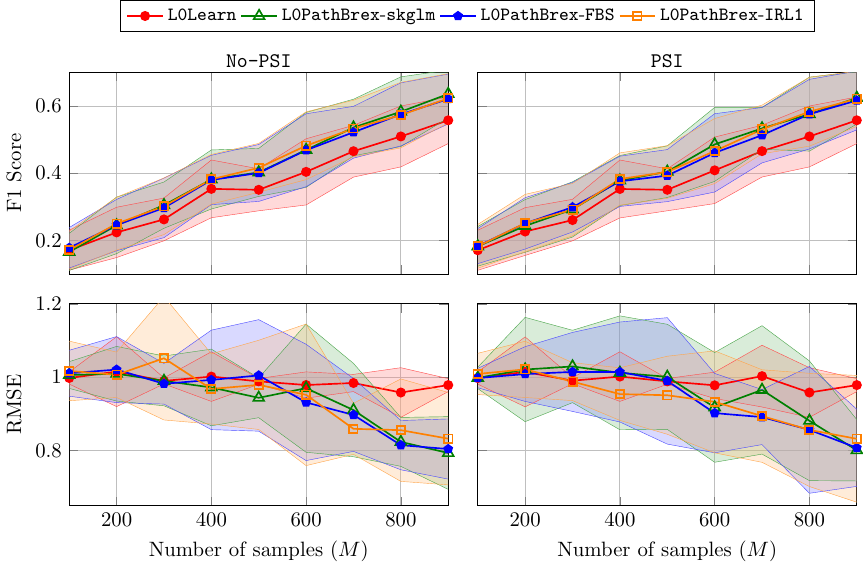}
\caption{
LR. Same as in Figure~\ref{fig:LR_Simu_Exp_Corr} with $k = 20$ and $\rho = 0.5$  for a constant correlation setting.
\label{fig:LR_Simu_Cst_Corr}
}
\end{figure}

\paragraph{Statistical performance as a function of number of samples.}

Model selection is evaluated as for the least-squares case, using the F1-score and RMSE metrics. The parameters $(N, s, \eta, k)$ are fixed to $(1000, 25, \eta, k)$, where $(\eta, k) = (0.9, 25)$ for exponential correlation and $(\eta, k) = (0.5, 20)$ for constant correlation. In addition, $\lambda_2$ is selected from a logarithmically spaced grid of 10 values in $[10^{-6}, 10]$, and the estimate $\hat{\x}$ is chosen as the one maximizing the F1-score over both the $\ell_0$-path and this $\lambda_2$ grid.

Figures~\ref{fig:LR_Simu_Exp_Corr} and~\ref{fig:LR_Simu_Cst_Corr} report the results for the exponential and constant correlation settings, respectively. As in the least-squares case, the proposed method~\Algo{} achieves better recovery performance in terms of both F1-score and RMSE compared to \Algolearn{}, for all choices of the inner solver~$\mathcal{A}$ with 4 passes. In addition, the local search strategy~\texttt{PSI} consistently improves the results across all methods.

We also observe that the choice of inner solver has only a minor impact on the statistical performance. However, in terms of computational cost, \texttt{FBS} is the fastest, followed by \texttt{skglm} and \texttt{IRL1}, each requiring a few seconds. When combined with \texttt{PSI}, \texttt{FBS} remains below 10 seconds, while \texttt{skglm} and \texttt{IRL1} typically exceed 10 seconds. In contrast, \Algolearn{} requires between one and a few seconds. 
The relatively higher computational cost of \Algolearn{} in the logistic regression compared to the least-squares case, can be explained by the absence of a closed-form solution for the one-dimensional subproblems of the form $u \mapsto J_0(\x^{(n)} + u \e_n)$, which are required in each coordinate descent update.

\paragraph{Pareto front estimation on real datasets.} In  Figure~\ref{fig:LR-real-60seconds}, we present the estimated Pareto fronts for the real datasets listed in Table~\ref{tab:real-data}, with a time limit of one minute imposed on the \Algo{} variants. For this experiment, comparisons are made only with \texttt{L0Learn}, as \texttt{L0PD} is specific to LS problems.

\begin{figure}
\includegraphics[width=\linewidth]{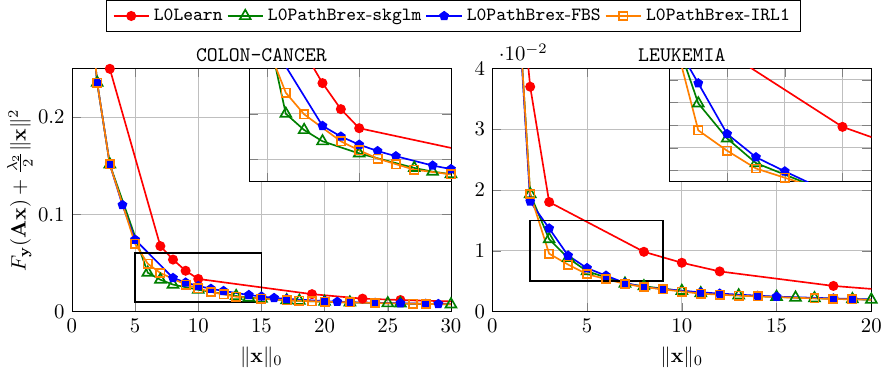}
\vspace{-0.5cm}
\caption{Estimated Pareto fronts for the LR problem on the two real datasets from Table~\ref{tab:real-data} with a one-minute time limit for all \Algo{} variants. The ridge regularization parameter is fixed to $\lambda_2 = 10^{-5}$.}\label{fig:LR-real-60seconds}
\end{figure}

As observed with LS experiments, all \Algo{} variants significantly outperform \texttt{L0Learn} on both datasets. With respect to the choice of inner algorithm $\mathcal{A}$, \texttt{skglm} is no longer consistently the best performer; for instance, \texttt{IRL1} achieves superior results on the \texttt{LEUKEMIA} dataset. Meanwhile, \texttt{FBS} consistently yields the lowest performance among the three.
We again supplement these results with Figures~\ref{fig:LR-real-15seconds} and~\ref{fig:LR-real-300seconds}, corresponding to time limits of 15 seconds and 5 minutes, respectively. For both datasets, we observe that i) after 15 seconds, all \Algo{} variants already outperform \texttt{L0Learn}, and ii) the estimated Pareto fronts continue to improve as additional passes are performed, up to 5 minutes.

\subsection{Sparse Kullback-Leibler regression}\label{sec:expeKL}
As mentioned earlier,  neither existing \Algolearn{} nor \texttt{L0PD} solvers supports KL-dependent problems. We thus leverage these KL experiments to compare the quality of the Pareto fronts estimated by \AlgoCR{\texttt{IRL1}} using either $\Psi_{\ell_2}$ or $\Psi_{\mathrm{KL}}$ to generate \BR{}, as defined in Section~\ref{sec:base_algos}. To this end, we generate 100 sparse Poisson regression problems with parameters $(M, N, \varrho, k) = (50, 100, 30, 10)$ and $\mathbf{b} = b \mathbf{1}$, where $b = \max(\A \mathbf{x}^*) / 100$. The matrix $\A$ is generated using the exponential correlation model. Finally, \AlgoCR{\texttt{IRL1}} is run until no new points remain to be explored in $\mathcal{X}$.
The results of each run are then classified into four categories:
\begin{itemize}
    \itemsep0pt
    \item \textbf{$\Psi_{\mathrm{KL}}$-best:} The Pareto front estimated with $\Psi_{\mathrm{KL}}$ is below the one estimated with~$\Psi_{\ell_2}$.
    \item \textbf{$\Psi_{\ell_2}$-best:} The Pareto front estimated with $\Psi_{\ell_2}$ is below the one estimated with $\Psi_{\mathrm{KL}}$.
    \item \textbf{Equivalent:} Both estimated Pareto fronts coincide.
    \item \textbf{Indecisive:} Estimated Pareto fronts intersect at least once, indicating that $\Psi_{\mathrm{KL}}$ yields better solutions for some sparsity levels, while $\Psi_{\ell_2}$ performs better for others.
\end{itemize}

Occurrences of each case are reported in Table~\ref{tab:KL_res} for three correlation parameters $\eta \in \{0.4, 0.6, 0.8\}$. While the majority of generated problems result in either equivalent ($\approx 50\%$) or indecisive ($\approx 25\%$) estimated Pareto fronts for both \BR{}, we observe an advantage of $\Psi_{\mathrm{KL}}$ in the remaining cases. This aligns with the fact that $\Psi_{\mathrm{KL}}$ is better suited than $\Psi_{\ell_2}$ to the geometry of the KL data fidelity term in order to derive an exact relaxation. In particular, the \BR{} associated with $\Psi_{\mathrm{KL}}$ has been previously showed to eliminate more local minimizers~\citep{essafri2024l0}.

\begin{table}[t]
    \centering
    \begin{tabular}{ccccc}
    \toprule
         & $\Psi_{\mathrm{KL}}$-best &  $\Psi_{\ell_2}$-best &  Equivalent & Indecisive \\
         \midrule  \midrule
        $ \eta = 0.4$ & 19 & 5 & 44 & 32 \\
        $ \eta = 0.6$ & 25 & 4 & 47 & 24 \\
        $ \eta = 0.8$ & 21 & 5 & 49 & 25  \\
        \bottomrule
    \end{tabular}
    \caption{Comparisons of estimated Pareto fronts for a Poisson regression problem (KL fidelity) solved by  \AlgoCR{\texttt{IRL1}} using either $\Psi_{\ell_2}$ or $\Psi_{\mathrm{KL}}$ to generate \BR{}. Reported values correspond to the occurrence of the four considered categories among 100 instances of the problem. }
    \label{tab:KL_res}
\end{table}

\section{Conclusions \& Outlook}

In this work, we proposed \Algo{} to  estimate the solution path of $\ell_0$-regularized optimization problems with general (i.e., potentially non-quadratic) data terms. The method exploits properties of exact relaxations $J_\Psi$ of  the original function $J_0$, obtained by replacing the $\ell_0$ term with $B_\Psi$, non-convex continuous relaxation defined in terms of Bregman divergences, called  \BR{}. Specifically, while each local minimizer $\hat{\x}$ of $J_0$ remains a local minimizer for any value of the regularization parameter $\lambda_0$, it corresponds to a local minimizer of $J_\Psi$ only over an interval $[\underline{\lambda}_0(\hat{\x}), \bar{\lambda}_0(\hat{\x})]$ of $\lambda_0$ values. Given any off-the-shelf algorithms able to deal with $J_\Psi$, \Algo{} leverages this range to implement warm-start strategies in both forward (exploring the lower bound $\underline{\lambda}_0(\hat{\x})$) and backward (exploring the upper bound $\bar{\lambda}_0(\hat{\x})$) directions.

We benchmarked the performance of the proposed \Algo{} against the state-of-the-art \texttt{L0Learn} algorithm method for both sparse least-squares and logistic regression problems. The comparisons were conducted from both statistical and optimization (Pareto front estimation) perspectives, using synthetic and real data, respectively. Additionally, we included comparisons with the greedy \texttt{L0PD} method for Pareto front estimation on real datasets in the least-squares setting. In all cases, \Algo{} using either \texttt{FBS}, \texttt{IRL1}, or \texttt{skglm} as inner solver achieved the best results at the cost of a reasonable increase in computational burden, especially given the difficulty of such $\ell_0$ problems.

A notable feature of \Algo{} is its versatility. It can be deployed with various inner algorithms and can leverage different relaxations $J_\Psi$ associated potentially to different geometries described by the choice of $\Psi$. 

Several directions for future work emerge from our observations. In particular, while \texttt{skglm} often provided the best compromise between performance and computational cost in our experiments, no single inner solver consistently outperformed the others across all settings. Likewise, although $\Psi_\mathrm{KL}$ showed promising results in the KL experiments, its advantage was not systematic in every run. These findings motivate the development of enhanced versions of \Algo{} capable of adaptively combining the strengths of multiple optimization algorithms and exact relaxations $J_\Psi$, potentially leveraging their complementary ability to eliminate different local minimizers without incurring additional computational cost.


\section*{Acknowledgments}

M. Essafri and E. Soubies acknowledge financial support from the French National Research Agency (ANR) under project EROSION (ANR-22-CE48-0004) and from the Toulouse AI cluster ANITI (ANR-23-IACL-0002).  
L. Calatroni acknowledges financial support from the European Research Council (ERC) under grant MALIN (No. 101117133).

This work represents only the view of the authors. The European Commission and the other
organizations are not responsible for any use that may be made of the information it contains.


\appendix

\section{Proof of Theorem~\ref{th:min-loc-removed}}\label{app:th-min-loc}

\begin{lemma}\label{lem:alpha-increasing}
Let $\psi$ satisfy the conditions of Definition~\ref{def:brex}. Then, $d_\psi(0,\cdot) : \R_{\geq 0 } \to \R_{\geq 0 }$ is a one-to-one map. Its inverse is the function $\alpha : \R_{\geq 0} \to \R_{\geq 0}$ that, given $\lambda \geq 0$, defines the bound $\alpha(\lambda)$ at which \BR{} becomes constant in~\eqref{eq:brex-generic-formula}.
\end{lemma}
\begin{proof}
 First, \(d_{\psi}(0, 0)=0\). Second, Condition i. (strict convexity of \(\psi\))  ensures that \(d_{\psi}(0, \cdot)\) is increasing on \(\R_{\geq 0}\). Indeed, \(d_{\psi}(0, \cdot)' = \psi''\), which is positive on $\R_{>0}$. Third, from Condition ii., \(d_{\psi}(0,x) \to \infty\) as $x \to \infty$. Hence,  \(d_{\psi}(0, \cdot) : \R_{\geq 0} \to \R_{\geq 0}\) is a one-to-one map. The end of the statement follows from the definition of $\alpha$.
\end{proof}

\noindent
{\bf Proof of the theorem}.
Let $\hat{\x} \in \Cc^N$ be a local minimizer of \( J_\Psi \) and set $g_n = \psi_n'(0)-\left<\a_n ,\nabla F_\y\left(\A\hat{\x} \right) \right> $ for all $n \in \sigma^c(\hat \x)$. Then, we get from Proposition~\ref{prop:loc-min-jpsi} that
    \begin{align}
        &\forall n \in \sigma(\hat{\x}), \; |\hat{x}_n| > \alpha_n(\lambda_0), \label{eq:condition1} \\
        & \forall n \in \sigma^c(\hat{\x}), \; \varphi_n\left( g_n\right) \leq \psi_n'\left(\alpha_n(\lambda_0)\right). \label{eq:condition2}
    \end{align}
    where we recall that $\varphi_n = |\cdot|$ when $\Cc = \R$ and $\varphi_n = \max(\cdot, \psi_n'(0))$ when $\Cc = \R_{\geq 0}$. Moreover, by differentiability and strict convexity of  $\psi_n$, we have that $\psi_n' : \Cc \to \mathrm{Range}(\psi_n')$ is bijective and is thus invertible on its range. Moreover, from Condition iii. of Definition~\ref{def:brex}, there exists $ L \in \R_{\geq 0} \cup \{+\infty\}$ such that 
    \begin{itemize}
    \itemsep0pt
        \item  $\mathrm{Range}(\psi_n') = (-L,L) $ if $\Cc = \R$,
        \item $\mathrm{Range}(\psi_n') = [\psi_n'(0),L) $ if $\Cc = \R_{\geq 0}$.
    \end{itemize}
    Hence $(\psi_n')^{-1}(\varphi_n(g_n)) \in \R_{\geq 0}$  is well defined as, from~\eqref{eq:condition2}, we get that  $\varphi_n\left( g_n\right) \in \mathrm{Range}(\psi_n')$.
    
    From this together with Lemma~\ref{lem:alpha-increasing}, we can rewrite \eqref{eq:condition1}-\eqref{eq:condition2} as
    \begin{align}
        &\forall n \in \sigma(\hat{\x}), \; d_{\psi_n}(0,|\hat{x}_n|) > \lambda_0, \label{eq:condition1-2} \\
        & \forall n \in \sigma^c(\hat{\x}), \; d_{\psi_n}\left(0,(\psi_n')^{-1} \left(\varphi_n\left( g_n\right) \right) \right)\leq \lambda_0. \label{eq:condition2-2}
    \end{align}
 Taking the min (resp. the max) over $n$ in~\eqref{eq:condition1-2} (resp.~\eqref{eq:condition2-2}) completes the proof.

\section{Additional Figures}

This appendix includes the complementary Figures~\ref{fig:LS-real-15seconds} to~\ref{fig:LR-real-300seconds}.

\begin{figure}[!t]
\centering
\includegraphics[width=\linewidth]{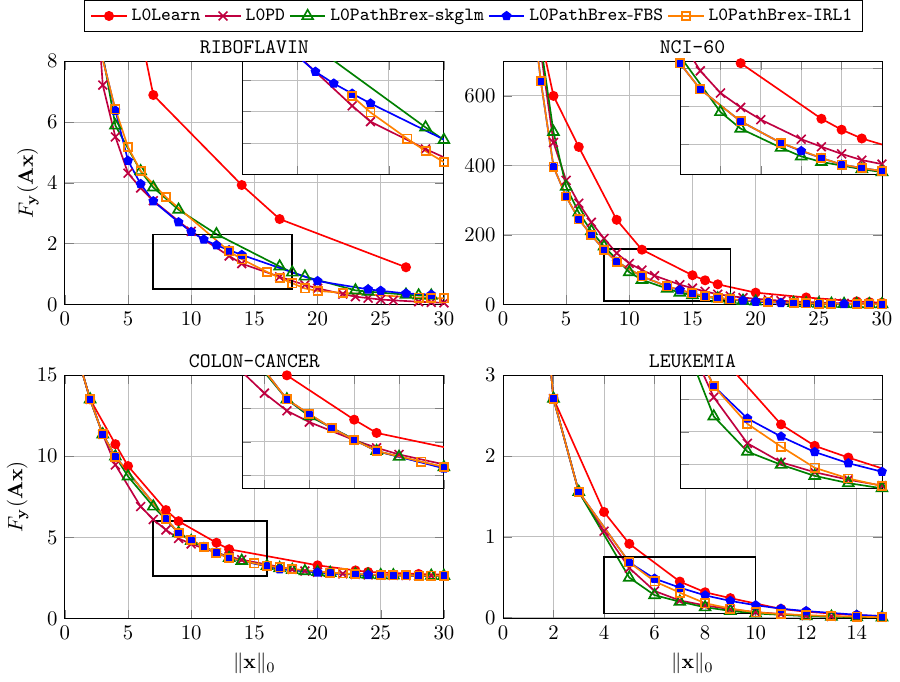}
\vspace{-0.5cm}
\caption{Same as Figure~\ref{fig:LS-real-60seconds} with a time limit of 15 seconds for the \Algo{} variants.
}
\label{fig:LS-real-15seconds}
\end{figure}

\begin{figure}[!t]
\centering
\includegraphics[width=\linewidth]{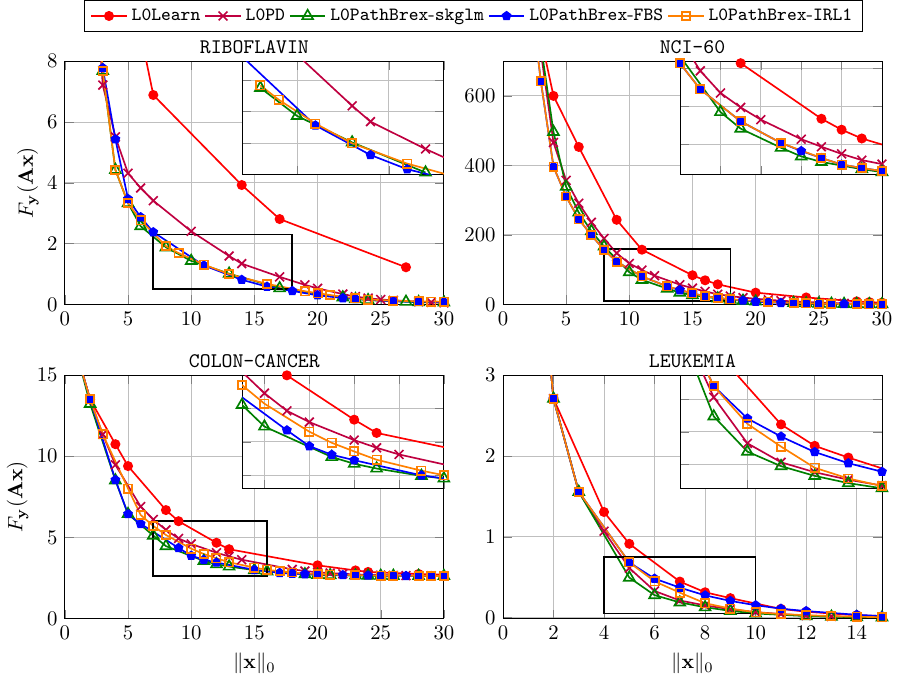}
\vspace{-0.5cm}
\caption{
Same as Figure~\ref{fig:LS-real-60seconds} with a time limit of 5 minutes for the \Algo{} variants.
}
\label{fig:LS-real-300seconds}
\end{figure}

\begin{figure}[!t]
\includegraphics[width=\linewidth]{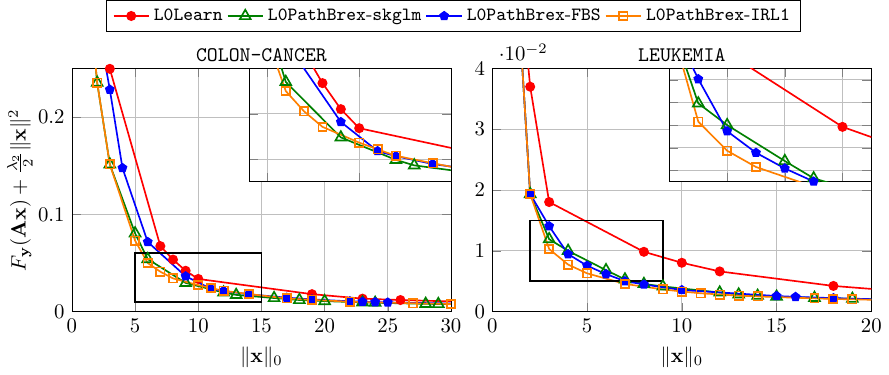}
\vspace{-0.5cm}
\caption{Same as Figure~\ref{fig:LR-real-60seconds} with a time limit of 15 seconds for the \Algo{} variants.}\label{fig:LR-real-15seconds}
\end{figure}

\begin{figure}[!t]
\includegraphics[width=\linewidth]{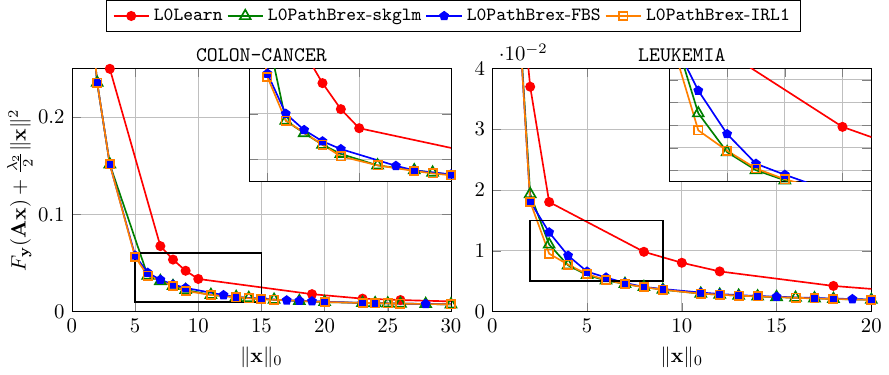}
\vspace{-0.5cm}
\caption{Same as Figure~\ref{fig:LR-real-60seconds} with a time limit of 5 minutes for the \Algo{} variants.}\label{fig:LR-real-300seconds}
\end{figure}

\bibliographystyle{plainnat}
\bibliography{arxiv}

\end{document}

%% file: newcommand.tex
\newcommand{\R}{\mathbb{R}}
\newcommand{\N}{\mathbb{N}}

  {\end{list}}

\makeatletter

  \gdef\listctr{list\romannumeral\the\@listdepth}\expandafter
  
\makeatother

\newcommand{\ve}[1]{\mathbf{#1}}
\def\x{\ve{x}}
\def\y{\ve{y}}

\def\z{\ve{z}}
\def\e{\ve{e}}

\def\a{\ve{a}}

\def\A{\ve{A}}
\def\Cc{\mathcal{C}}

\def\R{\mathbb R}

\newcommand{\argmin}{\operatornamewithlimits{arg}\operatornamewithlimits{min}}

\newcommand{\BR}{B-rex}
\newcommand{\AlgoCR}[1]{\texttt{L0PathBrex}-#1}
\newcommand{\Algo}{\texttt{L0PathBrex}}
\newcommand{\Algolearn}{\texttt{L0Learn}}